\documentclass[reqno]{amsart}
\usepackage{amsmath,amsfonts}
\usepackage{color}
\usepackage{hyperref}
\usepackage[inner=1.0in,outer=1.0in,bottom=1.0in, top=1.0in]{geometry}
\newtheorem{theorem}{Theorem}[section]
\newtheorem{conjecture}[theorem]{Conjecture}
\newtheorem{proposition}[theorem]{Proposition}
\newtheorem{corollary}[theorem]{Corollary}
\newtheorem{lemma}[theorem]{Lemma}
\newtheorem{remark}[theorem]{Remark}

\numberwithin{equation}{section}
\begin{document}
\title[Quantitative quantum ergodicity and the nodal domains]{Quantitative quantum ergodicity and the nodal domains of Hecke--Maass cusp forms}
\author{Junehyuk Jung}

\begin{abstract}
We prove a quantitative statement of the quantum ergodicity for Hecke--Maass cusp forms on the modular surface. As an application of our result, along a density $1$ subsequence of even Hecke--Maass cusp forms, we obtain a sharp lower bound for the $L^2$-norm of the restriction to a fixed compact geodesic segment of $\eta=\{iy~:~y>0\} \subset \mathbb{H}$. We also obtain an upper bound of $O_\epsilon\left(t_\phi^{3/8+\epsilon}\right)$ for the $L^\infty$ norm along a density $1$ subsequence of Hecke--Maass cusp forms; for such forms, this is an improvement over the upper bound of $O_\epsilon\left(t_\phi^{5/12+\epsilon}\right)$ given by Iwaniec and Sarnak.

In a recent work of Ghosh, Reznikov, and Sarnak, the authors proved for all even Hecke--Maass forms that the number of nodal domains, which intersect a geodesic segment of $\eta$, grows faster than $t_\phi^{1/12-\epsilon}$ for any $\epsilon>0$, under the assumption that the Lindel{\"o}f Hypothesis is true and that the geodesic segment is long enough. Upon removing a density zero subset of even Hecke--Maass forms, we prove without making any assumptions that the number of nodal domains grows faster than $t_\phi^{1/8-\epsilon}$ for any $\epsilon>0$.
\end{abstract}
\thanks{
We would like to thank Peter Sarnak for introducing his recent paper with Ghosh and Reznikov to the author, and suggesting this problem as a part of the Ph.D. thesis of the author. We also appreciate Peter Sarnak and Nicolas Templier for encouragement and many helpful comments. This work was supported by the National Research Foundation of Korea(NRF) grant funded by the Korea government(MSIP)(No. 2013042157), and partially by the National Science Foundation under agreement No. DMS-1128155. The author was also partially supported by TJ Park Post-doc Fellowship funded by POSCO TJ Park Foundation.}
\maketitle
\section{Introduction}

Let $\Gamma=SL_2\left(\mathbb{Z}\right)$ and let $\mathbb{X}$ be the modular surface $ \Gamma \backslash \mathbb{H}$. Let $\phi$ be an $L^2$-normalized Hecke--Maass cusp form on $\mathbb{X}$. In other words, $\phi$ is a joint eigenfunction of $-\Delta_\mathbb{X}$, the Laplace--Beltrami operator on $\mathbb{X}$, and the Hecke operators $\left\{T_n\right\}$, which are defined by the following action,
\[
T_n f\left(z\right) = \frac{1}{\sqrt{n}}\sum_{ad=n}\sum_{b~\left(d\right)}f\left(\frac{az+b}{d}\right),
\]
for $SL\left(2,\mathbb{Z}\right)$ invariant function $f$ on $\mathbb{H}$. We denote by $\lambda_\phi=1/4+t_\phi^2$ the eigenvalue of $-\Delta_\phi$ and by $\lambda_\phi\left(n\right)$ the eigenvalue of $T_n$ (the $n$th Hecke eigenvalue) corresponding to $\phi$. Note that $\left\{-\Delta_\mathbb{X}, T_1, T_2, \ldots\right\}$ is a commuting family of self-adjoint operators, where each joint eigenspace has dimension one by the Multiplicity One Theorem \cite{MR0401654}.

Let $\sigma : \mathbb{X} \to \mathbb{X}$ be an orientation reversing isometric involution induced by $x+iy \mapsto -x+iy$ on $\mathbb{H}$. Then one can check that $\sigma$ commutes with all of $\left\{-\Delta_\mathbb{X}, T_1, T_2, \ldots\right\}$, from which we infer that a Hecke--Maass cusp form on $\mathbb{X}$ is automatically an eigenfunction of $\sigma$. We say a Hecke--Maass cusp form $\phi$ is even (resp. odd) if $\sigma\phi=\phi$ (resp. $\sigma\phi = -\phi$).

\subsection{Number of nodal domains of even Maass--Hecke cusp forms}
Let $Z_\phi$ be the zero set of $\phi$, which in turn is a finite union of real analytic curves. For any subset $C \subseteq \mathbb{X}$, let $N^C\left(\phi\right)$ be the number of connected components (the nodal domains) in $\mathbb{X} \backslash Z_\phi$, that intersect $C$. Let $N\left(\phi\right)=N^\mathbb{X}\left(\phi\right)$.

In \cite{bs}, the authors estimate the expected number of nodal domains of random waves using a percolation like model. In view of Berry's conjecture that eigenfunctions having large eigenvalues are well modeled by random waves, results in \cite{bs} suggests the existence of a constant $c>0$ such that
\begin{equation}\label{nodal}
N\left(\phi\right)= c \lambda_\phi+o\left(\lambda_\phi\right).
\end{equation}
In \cite{ns}, the authors examined \eqref{nodal} for random spherical harmonics, and they proved the existence of $c>0$ such that \eqref{nodal} holds almost surely as $\lambda_\phi \to +\infty$.

Note that it is not true for a general Riemannian surface that the number of nodal domains of an eigenfunction must increase with the eigenvalue \cite{st,lewy}. In \cite{gzs}, the authors studied nodal domains crossing $\delta=\left\{iy~:~y>0\right\}$ and proved
\begin{equation}\label{capt}
t_\phi \ll N^\delta\left(\phi\right) \ll t_\phi \log t_\phi
\end{equation}
for even Hecke--Maass cusp forms $\phi$. Assuming that \eqref{nodal} is true, this estimate in particular implies that almost all nodal domains do not touch $\delta$!\footnote{In \cite{gzs}, nodal domains that intersect $\delta$ are called ``split."}

In order to prove the lower bound in  \eqref{capt}, the authors analyzed the region near the cusp determined by $y > t_\phi/100$. For nodal domains intersecting a fixed geodesic segment, they proved:
\begin{theorem}[\cite{gzs}]\label{thm3}
Let $\beta \subset \delta$ be a fixed compact geodesic segment that is sufficiently long. Assume the Lindel{\"o}f Hypothesis for the $L$-functions $L\left(s,\phi\right)$. Then
\[
N^\beta\left(\phi\right) \gg_\epsilon t_\phi^{\frac{1}{12}-\epsilon}.
\]
\end{theorem}
Recently in \cite{jung5}, $\lim_{t_\phi \to \infty} N^\beta\left(\phi\right) = +\infty$ is established without any assumptions. However no quantitative lower bound is given in \cite{jung5}. Here we present a stronger unconditional lower bound for almost all $\phi$ as an application of a sharp estimate on the variance of the shifted convolution sums (Theorem \ref{thm}).
\begin{theorem}\label{thm6}
Let $\beta \subset \delta$ be any fixed compact geodesic segment. Fix $\epsilon>0$. Then all but $O\left(T^{\frac{4}{3}-\frac{\epsilon}{2}}\right)$ forms within the set of even Hecke--Maass cusp forms in $\left\{\phi~:~|T-t_\phi|<T^{1/3}\right\}$ satisfy
\[
N^\beta\left(\phi\right)>t_\phi^{\frac{1}{8}-\epsilon}.
\]
\end{theorem}
\begin{remark}
For a negatively curved surface with an isometric involution, the existence of a density $1$ subsequence of even eigenfunctions having a growing number of nodal domains is established in \cite{jung4}, however, without an explicit lower bound. Recently in \cite{zeld1}, a logarithmic lower bound for the number of nodal domains is obtained in the same setting; one of the main ingredients is a logarithmic improvement over quantum ergodicity theorem in \cite{heza}.
\end{remark}

\subsection{Quantitative Quantum Ergodicity and shifted convolution sums}
Because $\phi$ is invariant under the action by $\Gamma_\infty :=\left\{ \left( \begin{array}{cc}
1 & n \\
0 & 1  \end{array} \right) ~:~n\in \mathbb{Z}\right\} \subset \Gamma$, one can consider the Fourier expansion of $\phi$ at the cusp $\infty$, and it is given by
\[
\phi\left(z\right)=\sqrt{\cosh \left(\pi t_\phi\right)}\sum_{n \neq 0 } \rho_\phi\left(n\right) \sqrt{y} K_{it_\phi} \left(2\pi |n|y\right)e\left(nx\right).
\]
Here $K_{it}\left(y\right)$ is the modified Bessel function of the second kind and $e\left(x\right) = e^{2\pi i  x}$.

Note that the Fourier coefficients satisfy the relation $\rho_\phi\left(\pm n\right)=\rho_\phi\left(\pm 1\right)\lambda_\phi\left(n\right)$ for $n>0$ and $\rho_\phi\left(1\right) = \pm \rho_\phi\left(-1\right)$ depending on the parity of $\phi$. The first fourier coefficient $\rho_\phi\left(1\right)$ is known to satisfy the estimate
\begin{equation}\label{triv}
t_\phi^{-\epsilon} \ll_\epsilon |\rho_\phi\left(1\right)| \ll_\epsilon t_\phi^\epsilon \footnote{Here and elsewhere, $A \ll_\omega B$ means $|A|<CB$ for some constant $C$ depending only on $\omega$. We also use the notation $A\sim_\omega B$ for $B \ll_\omega A \ll_\omega B$.}
\end{equation}
for any $\epsilon>0$ \cite{iwan90,hofloc}. From the recurrence relation of $T_n$, we also know that,
\begin{equation}\label{heck}
\lambda_\phi(n)\lambda_\phi(m) = \sum_{d|\left(n,m\right)} \lambda_\phi\left(\frac{nm}{d^2}\right).
\end{equation}
These are the main arithmetic inputs that we are going to use to analyze $\phi$.

The shifted convolution sum of Fourier coefficients is given by
\[
\frac{1}{t_\phi}\sum_n \rho_\phi\left(n+m\right)\rho_\phi\left(n\right) \psi\left(\frac{\pi|n|}{t_\phi}\right)
\]
for a fixed $m \in \mathbb{Z}$ and a test function $\psi \in C_0^\infty \left(0,\infty\right)$. Note that these type of sums appear in the expansion of
\[
\int_\mathbb{X} P_{m,h}\left(z\right) |\phi\left(z\right)|^2 dV,
\]
where $P_{m,h}\left(z\right)$ is the Poincar\'e series corresponding to $m\in \mathbb{Z}$ and $h \in C_0^\infty \left(0,\infty\right)$ given by
\[
P_{m,h}(z)=\sum_{\gamma \in \Gamma_\infty \backslash \Gamma} h(\mathrm{Im}(\gamma z))e^{2\pi i m \mathrm{Re}(\gamma z)},
\]
and $dV$ is the area measure induced from $y^{-2}dxdy$ on $\mathbb{H}$ to $\mathbb{X}$. (See Theorem \ref{app}.) Since the space of the Poincar\'e series spans $L^2\left(\mathbb{X}\right)$, one can study the weak limit of the measure
\[
d\mu_\phi = |\phi\left(z\right)|^2 dV
\]
as $t_\phi \to \infty$ via estimating the shifted convolution sums. For instance, it is proven in \cite{sarluo2} that the estimate
\begin{equation}\label{que2}
\left|\frac{1}{t_\phi}\sum_n \rho_\phi\left(n+m\right)\rho_\phi\left(n\right) \psi\left(\frac{\pi|n|}{t_\phi}\right) - \frac{24}{\pi^3} \delta_{0,m}\int_0^\infty \psi\left(y\right)dy\right| = o\left(1\right)
\end{equation}
as $t_\phi \to \infty$ implies
\begin{equation}\label{que1}
\left|\int_\mathbb{X} f d\mu_\phi - \int_\mathbb{X} f dV \right| = o\left(1\right)
\end{equation}
as $t_\phi \to \infty$, for any $f \in C_0^\infty \left(\mathbb{X}\right)$. Note that \eqref{que1} is a consequence of the arithmetic Quantum Unique Ergodicity (QUE) theorem of Lindenstrauss \cite{lin06} and Soundararajan \cite{so10}.
\begin{remark}
In fact, \eqref{que2} is equivalent to the arithmetic QUE theorem. In order to establish the equivalence, one has to make use of full QUE, in other words, QUE with the symbols depending on the phase space. Since we only discuss QUE for the base measure in this article, we omit the proof of the equivalence.
\end{remark}

Quantitative QUE conjecture concerns the rate of convergence of \eqref{que1}, which again can be obtained if we have a quantitative version of \eqref{que2}. We call this the strong Quantitative QUE conjecture:
\begin{conjecture}[strong Quantitative QUE conjecture]\label{qque}
There exist $\nu>0$ and $k<\infty$ such that for any $\psi \in C_0^\infty \left(0,\infty\right)$,
\begin{equation}\label{eq1.1}
\left|\frac{1}{t_\phi}\sum_n \rho_\phi\left(n+m\right)\rho_\phi\left(n\right) \psi\left(\frac{\pi|n|}{t_\phi}\right) - \frac{24}{\pi^3} \delta_{0,m}\int_0^\infty \psi\left(y\right)dy\right| \ll_m t_\phi^{-\nu}\|\psi\|_{W^{k,\infty}\left(0,\infty\right)}.
\end{equation}
\end{conjecture}
Here $\|\cdot\|_{W^{k,\infty}\left(0,\infty\right)}$ is a Sobolev norm defined by
\[
\|f\|_{W^{k,\infty}\left(0,\infty\right)} = \sum_{j=0}^k \sup_{x \in \left(0,\infty\right)} \left|\partial_x^j f\left(x\right)\right|
\]
for $f \in C_0^\infty \left(0,\infty\right)$.
\begin{remark}
The estimate \eqref{eq1.1} for all $\nu<\frac{1}{2}$ implies Lindel{\"o}f Hypothesis for the central value of the triple product $L$-function $L\left(1/2,\phi\times\phi\times\phi_0\right)$ where $\phi_0$ is a fixed Hecke--Maass cusp form $\phi_0$ \cite{wa02}.
\end{remark}

In this article, we estimate the variance of shifted convolution sums over the range $T-G < t_\phi < T+G$ where $G$ is assumed to be a small power of $T$.
\begin{theorem}\label{thm}
Let $\theta$ and $\epsilon$ be fixed constants satisfying $1/3 < \theta <1$ and $\epsilon>0$, respectively. Assume that $\psi \in C_0^\infty\left(0,\infty\right)$ is supported on $\left(1/l,l\right) \subset\left(0,\infty\right)$ for some fixed $l>0$. Let $X$ be a parameter satisfying $1 \ll X \ll T$. Then there exists $A>0$ depending only on $\theta$ and $\epsilon$ such that
\begin{multline}\label{eq1.2}
\sum_{|t_\phi-T|<T^\theta}\left| \sum_n \rho_\phi\left(n+m\right)\rho_\phi\left(n\right) \psi\left(\frac{\pi n}{X}\right)-\delta_{0,m}\frac{12X}{\pi^3} \int_0^\infty \psi\left(y\right)dy\right|^2 \\
\ll_{\epsilon,\theta,  l} (|m|^{3/2}+1) XT^{1+\theta+\epsilon}\|\psi\|_{W^{A,\infty}}^2
\end{multline}
holds uniformly in $|m|<X^{\frac{1}{2}}$. One can take, for example, $A=\max\{100/(3\theta-1),300/\epsilon\}$.
\end{theorem}
\begin{remark}
For the holomorphic Hecke eigenforms, this kind of average over a short range is studied in \cite{sarluo}. Theorem \ref{thm} is a generalization to Hecke--Maass eigenforms.
\end{remark}

According to Weyl's law \cite{zbMATH02629881}, there are asymptotically $\sim T^{1+\theta}$ Hecke--Maass cusp forms in
\[
\left\{\phi~:~T<t_\phi<T+T^\theta\right\},
\]
for $0 \leq \theta \leq 1$. Hence Theorem \ref{thm} implies that the strong Quantitative QUE conjecture is true for almost all Hecke--Maass cusp forms. In particular, we obtain Quantitative Quantum Ergodicity in a short range $T-T^{1/3}<t_\phi < T+T^{1/3}$.
\begin{corollary}[Quantitative Quantum Ergodicity]\label{cor122}
Fix $\epsilon>0$ and let $h\in C_0^\infty(0,\infty)$ be a function supported in $(1/L,L)$ for some $L>1$. Then there exists a sufficiently small $\kappa>0$ and a sufficiently large $A>0$, both depending only on $\epsilon>0$, such that
\[
\sum_{|t_\phi-T| <T^{1/3}}\left|\int_\mathbb{X} P_{m,h}\left(z\right) |\phi\left(z\right)|^2 dV- \frac{3}{\pi}\int_{\mathbb{X}} P_{m,h}(z) dV\right|^2
\ll_{\epsilon_0,L} T^{1/3+\epsilon}\|h\|_{W^{A,\infty}},
\]
holds uniformly in $h$ and $|m|<T^\kappa$.
\end{corollary}
\begin{remark}
Quantitative Quantum Ergodicity for Maass--Hecke cusp forms is first proved in \cite{zel92}, where the following is obtained
\begin{equation}\label{log}
\sum_{t_\phi<T}\left|\int_\mathbb{X} f d\mu_\phi-\frac{3}{\pi}\int_\mathbb{X} f dV\right|^2 = O_f\left( \frac{T^2}{\log T}\right)
\end{equation}
for any $f \in C_0^\infty (\mathbb{X})$. The error bound is improved in \cite{sarluo2} to $O_{f,\epsilon}\left(T^{1+\epsilon}\right)$, which is essentially optimal.

In \cite{jak97}, the author considered a similar average with the microlocal lift $d\omega_\phi$ (the Wigner distribution) of $d\mu_\phi$ to the unit cotangent bundle $S^*\mathbb{X}$, and proved
\[
\sum_{t_\phi<T}\left|\int_{S^*\mathbb{X}} f d\omega_\phi-\int_{S^*\mathbb{X}} f d\omega\right|^2 = O_{f,\epsilon}\left(T^{1+\epsilon}\right).
\]
Here $d\omega$ is the Liouville measure on $S^*\mathbb{X}$.

In \cite{zhao10}, the author obtained a precise asymptotic for the above sum weighted by the first Fourier coefficients: for any $\epsilon>0$ and any $f\in C_0^\infty (\mathbb{X})$,
\[
\sum_{t_\phi} h_{T,T^{1-\epsilon}}(t_j)\frac{2}{\rho_\phi(1)^{2}}\left|\int_\mathbb{X} f d\mu_\phi -\frac{3}{\pi}\int_\mathbb{X} f dV\right|^2 = T^{1-\epsilon}V(f,f)+O\left(T^{\frac{1}{2}+\epsilon}\right),
\]
where $V$ is a non-negative Hermitian form that can be computed explicitly. (Here $h_{T,G}(y)$ is the function defined in Lemma \ref{kuz}.) See \cite{qv1} for the Quantum Variance for holomorphic modular forms.

In \cite{qqrm1}, the estimate \eqref{log} is generalized to an orthonormal eigenbasis on any given Riemannian manifold, and the author achieved a logarithmic saving over the trivial bound. This result \cite{qqrm1} is further generalized to quantizations of symplectic maps on tori in \cite{qqrm2}.
\end{remark}
\begin{remark}
As noted above, Corollary \ref{cor122} implies that Lindel{\"o}f Hypothesis holds for the triple product $L$-functions on the shorter range $T-T^{1/3+\epsilon}<t_\phi <T+T^{1/3+\epsilon}$ compared to the longer range established in \cite{sarluo2}.
\end{remark}

\subsection{\texorpdfstring{$L^p$}{Lp} restrictions}
Let $\beta \subset \left\{iy~:~y>0\right\}$ be a compact geodesic segment on $\mathbb{X}$. In \cite{gzs}, as an application of the arithmetic QUE theorem, a lower bound for the $L^2$ restriction is obtained:
\begin{equation*}
\int_\beta |\phi\left(z\right)|^2 ds \gg_\beta 1,
\end{equation*}
when $\phi$ is an even Hecke--Maass cusp form, under the assumption that $\beta$ is sufficiently long. The authors of \cite{gzs} also noted that, Conjecture \ref{qque} allows one to remove the assumption on $\beta$ being sufficiently long. Therefore, we deduce from Theorem \ref{thm} a sharp lower bound for the $L^2$-norm of restriction to any fixed geodesic $\beta$, for almost all even Hecke--Maass cusp forms.
\begin{corollary}\label{cor112}
Let $\beta \subset \left\{iy~:~y>0\right\}$ be any fixed compact geodesic segment. Fix $\epsilon >0$. Then
\[
\int_\beta |\phi\left(z\right)|^2 ds \gg_{\beta,\epsilon} 1
\]
is satisfied for all but $O_{\beta,\epsilon}\left(T^{1/3+\epsilon}\right)$ forms within the set of even Hecke--Maass cusp forms in $$\left\{\phi~:~|T-t_\phi|<T^{1/3} \right\},$$ as $T \to \infty$.
\end{corollary}
\begin{remark}
This type of lower bound for the $L^2$-norm of the restriction in the context of Quantum Ergodicity is first proved in \cite{ctz}. In particular, the authors prove that $\int_\beta |\phi\left(z\right)|^2 ds$ tends to twice the length of $\beta$ along a subsequence of density $1$. For related results, we refer the reader to \cite{burq,hz,tz1,dz}.

In Corollary \ref{cor112}, we are using (strong) Quantitative Quantum Ergodicity, hence we further obtain an estimate for the number of exceptional forms.

Recently in an unpublished work by Hassell and Toth, and in \cite{jung5}, using the idea from \cite{ctz}, the authors proved Corollary \ref{cor112} for all even Hecke--Maass cusp forms. However, for the application to Theorem \ref{thm6}, it is sufficient to use Corollary \ref{cor112}.
\end{remark}
Another application of Theorem \ref{thm} concerns the $L^\infty$-norm of Hecke--Maass cusp forms. In \cite{iw}, using Selberg's trace formula and amplification method, a nontrivial improvement of the $L^\infty$-norm of a Hecke--Maass cusp form is achieved.
\begin{theorem}[\cite{iw}]\label{supnorm}
Let $\phi$ be a Hecke--Maass cusp form on $\mathbb{X}$. Then for any fixed compact subset $C$ of $\mathbb{X}$, we have
\[
\sup_{z \in C}|\phi\left(z\right)| \ll_{C,\epsilon} t_\phi^{\frac{5}{12}+\epsilon}.
\]
\end{theorem}
Observe that Theorem \ref{thm} allows one to study lower bounds for the partial sum $\sum_{n<X} |\rho_\phi (n)|^2$. Such lower bounds can be used to find better amplifiers, and as a result, we obtain an improvement over Theorem \ref{supnorm} for almost all Hecke--Maass cusp forms.
\begin{corollary}\label{cor2}
Fix a compact set $C \subset \mathbb{X}$ and a non-negative constant $\epsilon$. Then, all but $O_\epsilon\left(T^{\frac{13}{12}+\epsilon}\right)$ Hecke--Maass cusp forms $\phi$ with $|T-t_\phi|<T^{1/3}$ satisfy
\[
\sup_{z \in C}|\phi\left(z\right)| \ll_{C, \epsilon} t_\phi^{\frac{3}{8}+\epsilon},
\]
as $T \to \infty$.
\end{corollary}

\section{Outline of the proof and preliminary results}
In this section, we review some ingredients that will be used in subsequent sections. To simplify our notation, let $e_c(x)=\exp\left(\frac{2\pi i x}{c}\right)$. For any integers $m,n,$ and $c\neq 0$, the Kloosterman sum is given by
\[
S(m,n,c)=\sum_{\substack{x \pmod{c}\\ \gcd(x,c)=1}}e_c\left(mx+n\bar{x}\right),
\]
where $\bar{x}$ is a multiplicative inverse of $x$ modulo $c$. Because
\[
\overline{S(m,n,c)}=\sum_{\substack{x \pmod{c}\\ \gcd(x,c)=1}}e_c\left(-mx-n\bar{x}\right)= \sum_{\substack{-x \pmod{c}\\ \gcd(x,c)=1}}e_c\left(mx+n\bar{x}\right) = S(m,n,c),
\]
we know that the Kloosterman sum is real.

The first step of the proof of Theorem \ref{thm} is to use Kuznetsov trace formula to transform the variance of shifted convolution sums over the range $T-T^\theta < t_\phi < T+ T^\theta$ into an exponential sum involving Kloosterman sums. The Kuznetsov trace formula that we are going to use in this paper is the following.
\begin{lemma}[Kuznetsov trace formula \cite{kuz80}]\label{kuzn}
Let $h\left(y\right)=e^{-y^2}$. For any given $T,G>0$, let
\[
h_{T,G}\left(y\right)=h\left(\left(y-T\right)/G\right)+h\left(\left(y+T\right)/G\right).
\]
By letting $\tau(m,r)$ denote the Fourier coefficient for Eisenstein series $E\left(\cdot, \frac{1}{2}+ir\right)$, that is given by
\[
\tau(m,r)= m^{ir}\sum_{d|m} d^{-2ir},
\]
we have
\begin{multline}\label{kuz}
\sum_\phi \rho_\phi(m)\rho_\phi(n) \frac{h_{T,G}(t_\phi)}{\cosh \pi t_\phi} +\frac{1}{4\pi}\int_{-\infty}^\infty \tau(m,r)\tau(n,r) \frac{h_{T,G}(r)dr}{\cosh \pi r}\\
=\frac{\delta(m,n)}{\pi^2}\int_{-\infty}^\infty r h_{T,G}(r) \tanh (\pi r) dr+\frac{2i}{\pi}\sum_{c}\frac{S(m,n,c)}{c}g\left(\frac{4\pi\sqrt{mn}}{c}\right),
\end{multline}
where
\begin{equation}\label{inttrans}
g(x)=\int_{-\infty}^\infty J_{2ir}(x) \frac{rh_{T,G}(r)}{\cosh \pi r} dr.
\end{equation}
Here $J_{2ir}(x)$ is the Bessel function of the first kind.
\end{lemma}
\begin{remark}
$\tau(m,r)$ is real-valued, because
\[
\overline{\tau(m,r)}= m^{-ir} \sum_{d|m} d^{2ir}= m^{-ir} \sum_{d|m} \left(\frac{m}{d}\right)^{2ir} = m^{ir} \sum_{d|m} d^{-2ir} = \tau(m,r).
\]
\end{remark}
\begin{remark}
A broader family of test functions (see \cite{kuz80} or \cite{ant}) can be used instead of $h_{T,G}$ in the Kuznetsov trace formula. For our purpose, it is sufficient to consider $h_{T,G}$.
\end{remark}
The same approach is used in \cite{sarluo2}, where the authors applied Weil's bound \cite{weil} to each Kloosterman sum, and proved Theorem \ref{thm} for $\theta=1$.
\begin{lemma}[Weil's bound \cite{weil}]\label{weilb}
Denoting by $\tau(n)$ the number of divisors of $n$, we have
\[
|S(n,m,c)| \leq \left(\gcd(n,m,c)\right)^{\frac{1}{2}}c^{\frac{1}{2}}\tau(c).
\]
\end{lemma}
\begin{remark}\label{rem1}
$\tau(n)$ satisfies $\tau(n) \ll_\epsilon n^\epsilon$ for any $\epsilon>0$.
\end{remark}
Note that the same proof cannot be applied if $\theta<1$. In order to handle the case where $\theta<1$, we have to exploit extra cancellation coming from the sum of Kloosterman sums. To this end, we follow the idea that is used in \cite{sarluo}, where the authors prove an analogue of Theorem \ref{thm} for holomorphic Hecke eigenforms using the Petersson trace formula. The main difference in the proof is in the analytic part of the proof. In \cite{sarluo}, a sum of Bessel functions $J_k(x)$ over the interval $K-K^\theta<k<K+K^\theta$ for large $K$ is analyzed, whereas we study the integral transform \eqref{inttrans} in this paper.
\subsection{Bessel transform}
Assume that $G=T^\theta$ for some fixed $0<\theta<1$. In order to use Lemma \ref{kuzn}, we need to analyze the integral transform
\[
g\left(x\right)=\int_{-\infty}^\infty J_{2iy}\left(x\right) \frac{h_{T,G}\left(y\right)y}{\cosh \pi y}dy.
\]
in terms of $G$ and $T$.  We begin by collecting some facts about the Bessel function $J_{2iy}(x)$, that can be found in \cite{er81}.
\begin{proposition}
For $\nu \in \mathbb{C}$ and $x>0$, the Bessel function of the first kind is given by the series
\begin{equation}\label{bes:series}
J_{\nu} (x) = (x/2)^\nu \sum_{k=0}^\infty (-1)^k \frac{(x/2)^{2k}}{k! \Gamma (\nu+k+1)}.
\end{equation}
For $x>0$ and $y \in \mathbb{R}$,
\begin{equation}\label{bes:conj}
\overline{J_{2iy}(x)} = J_{-2iy}(x),
\end{equation}
and for $0<x<C$ and $y \in \mathbb{R}$,
\begin{equation}\label{bes:triv}
J_{2iy}(x)/\cosh \pi y \ll_C 1.
\end{equation}
\end{proposition}
\begin{proof}
We use \eqref{bes:series} to see that
\begin{align*}
\overline{J_{2iy}(x)}&= \overline{(x/2)^{2iy} \sum_{k=0}^\infty (-1)^k \frac{(x/2)^{2k}}{k! \Gamma (2iy+k+1)}}\\
&=(x/2)^{-2iy} \sum_{k=0}^\infty (-1)^k \frac{(x/2)^{2k}}{k! \Gamma (-2iy+k+1)}\\
&=J_{-2iy}(x).
\end{align*}
For \eqref{bes:triv}, note that
\[
|\Gamma(2iy+1)|^2= 4y^2|\Gamma(2iy)|^2 = \frac{4\pi y}{\sinh 2\pi y} \gg 1/(\cosh \pi y)^2,
\]
hence
\begin{align*}
J_{2iy}(x)/\cosh \pi y &\ll \left|J_{2iy}(x)\Gamma(2iy+1)\right|\\
&=\left|(x/2)^{2iy} \sum_{k=0}^\infty (-1)^k \frac{(x/2)^{2k}}{k! (2iy+1)(2iy+2)\ldots (2iy+k)}\right|\\
&\leq \sum_{k=0}^\infty  \frac{(x/2)^{2k}}{(k!)^2}\\
&\ll_C 1. \qedhere
\end{align*}
\end{proof}
Recall the uniform asymptotic expansion of $J_{2iy}(x)$ when $x>1$ from \cite{er81}.
\begin{proposition}\label{prop:asymp}
For $x>1$ and $y \in \mathbb{R}$,
\begin{multline*}
J_{2iy}\left(x\right)=\frac{1}{\sqrt{2}\pi}\left(4y^2+x^2\right)^{-1/4}\exp\left(i\sqrt{4y^2+x^2}-2iy \sinh^{-1}\left(2y/x\right)\right)\\
\times \cosh \left(\pi y\right) \left(\sum_{m=0}^{N-1} b_m \left(4y^2+x^2\right)^{-m/2}+ O\left(x^{-N}\right)\right),
\end{multline*}
where $b_m$ is a linear combination of
\[
c_{m,\kappa} = y^{2\kappa}(4y^2+x^2)^{-\kappa}
\]
for $\kappa=0,1,\ldots, m$. For instance,
\[
b_0=1,~ b_1 = -\frac{1}{8}+\frac{5}{24} (1+2^{-2}x^2y^{-2})^{-1},~ b_2=\frac{3}{128}-\frac{77}{576} (1+2^{-2}x^2y^{-2})^{-1}+\frac{385}{3456} (1+2^{-2}x^2y^{-2})^{-2}\ldots.
\]
\end{proposition}
From \eqref{bes:triv} and Proposition \ref{prop:asymp}, we know that $J_{2iy}(x)/\cosh \pi y$ is uniformly bounded in $x>0$ and $y\in \mathbb{R}$.
\begin{lemma}\label{lem1}
Let $K>0$ and $1>\theta> \epsilon>0$ be fixed constants. Then for any $A>0$,
\[
g\left(x\right) \ll_{\epsilon, A,K} T^{-A}
\]
holds uniformly in $0<x<KGT^{1-\epsilon}$, for all sufficiently large $T$.
\end{lemma}
\begin{proof}
Observe that
\[
g\left(x\right)=\int_{-\infty}^\infty J_{2iy}\left(x\right) \frac{h_{T,G}\left(y\right)y}{\cosh \pi y}dy=-\int_{-\infty}^\infty J_{-2iy}\left(x\right) \frac{h_{T,G}\left(y\right)y}{\cosh \pi y}dy,
\]
by the change of variable $y \to -y$, hence
\[
2g\left(x\right)=\int_{-\infty}^\infty (J_{2iy}(x)-J_{-2iy}(x)) \frac{h_{T,G}(y)y}{\cosh \pi y}dy= \int_{-\infty}^\infty \frac{J_{2iy}(x)-J_{-2iy}(x)}{\sinh \pi y} h_{T,G}(y)y\tanh \pi y dy.
\]
Note that
\begin{align*}
&\int_{-\infty}^\infty \frac{J_{2iy}(x)-J_{-2iy}(x)}{\sinh \pi y} h_{T,G}(y)(y\tanh \pi y-|y|) dy\\
=&\int_{|y\pm T|<T/2} \frac{J_{2iy}(x)-J_{-2iy}(x)}{\sinh \pi y} h_{T,G}(y)(y\tanh \pi y-|y|) dy\\
&+\int_{|y\pm T|\geq T/2} \frac{y(J_{2iy}(x)-J_{-2iy}(x))}{\sinh \pi y} h_{T,G}(y)(\tanh \pi y-\mathrm{sgn}(y)) dy\\
\ll& \int_{|y\pm T|<T/2} |y\tanh \pi y-|y|| dy + \int_{|y\pm T|\geq T/2} (|y|+1) h_{T,G}(y) dy\\
\ll_B& T^{-B},
\end{align*}
for any constant $B>0$.

Now we use the following formula from \cite{er54}
\begin{equation*}
\left(\frac{J_{2iu}\left(x\right)-J_{-2iu}\left(x\right)}{\sinh \pi u}\right)^\wedge \left(y\right)= -i \cos\left(x \cosh\left(\pi y\right)\right),
\end{equation*}
and the Plancherel theorem to deduce that
\begin{align*}
2g\left(x\right)&=\int_{-\infty}^\infty \frac{J_{2iy}\left(x\right)-J_{-2iy}\left(x\right)}{\sinh \pi y} h_{T,G}\left(y\right)|y|dy +O_B\left(T^{-B}\right)\\
&=-i\int_{-\infty}^\infty \cos\left(x \cosh\left(\pi y\right)\right)\left(h_{T,G}\left(u\right)|u|\right)^\wedge\left(y\right) dy+O_B\left(T^{-B}\right).
\end{align*}
To handle
\[
\left(h_{T,G}\left(u\right)|u|\right)^\wedge\left(y\right),
\]
we first note that
\[
\int_{-\infty}^0 h\left(\frac{u-T}{G}\right) u e(yu)du \ll_C T^{-C}
\]
and that
\begin{align*}
&\int_{-\infty}^0 h\left(\frac{u-T}{G}\right) u e(yu)du\\
= &-\frac{1}{2\pi i y}\int_{-\infty}^0 \left(\frac{uh'}{G} +h\right)\left(\frac{u-T}{G}\right) e(yu)du\\
=&-\frac{1}{4\pi^2 y^2} \int_{-\infty}^0 \left(\frac{uh''}{G^2}+\frac{2h'}{G} \right)\left(\frac{u-T}{G}\right) e(yu)du + \frac{h(T/G)}{4\pi^2 y^2}\\
\ll_C& T^{-C} y^{-2},
\end{align*}
for any $C>0$. Therefore
\[
\int_{-\infty}^\infty h\left(\frac{u-T}{G}\right)|u|e\left(yu\right)du = \int_{-\infty}^\infty h\left(\frac{u-T}{G}\right)ue\left(yu\right)du+O_C\left(T^{-C}(1+y^2)^{-1}\right),
\]
and likewise
\[
\int_{-\infty}^\infty h\left(\frac{u+T}{G}\right)|u|e\left(yu\right)du = -\int_{-\infty}^\infty h\left(\frac{u+T}{G}\right)ue\left(yu\right)du+O_C\left(T^{-C}(1+y^2)^{-1}\right).
\]
Combining these, we see that
\begin{align*}
\left(h_{T,G}\left(u\right)|u|\right)^\wedge\left(y\right)
=&\int_{-\infty}^\infty \left(h\left(\frac{u-T}{G}\right)+h\left(\frac{u+T}{G}\right)\right)|u|e\left(yu\right)du\\
=&\int_{-\infty}^\infty h\left(\frac{u-T}{G}\right)ue\left(yu\right)du-\int_{-\infty}^\infty h\left(\frac{u+T}{G}\right)ue\left(yu\right)du+O_C\left(T^{-C}\left(1+y^2\right)^{-1}\right)\\
=&\int_{-\infty}^\infty h\left(\frac{u-T}{G}\right)ue\left(yu\right)du+\int_{-\infty}^\infty h\left(\frac{u-T}{G}\right)ue\left(-yu\right)du+O_C\left(T^{-C}\left(1+y^2\right)^{-1}\right)\\
=&Ge\left(Ty\right)\left(h\left(u\right)\left(Gu+T\right)\right)^\wedge\left(Gy\right)+Ge\left(-Ty\right)\left(h\left(u\right)\left(Gu+T\right)\right)^\wedge\left(-Gy\right)\\
&+O_C\left(T^{-C}\left(1+y^2\right)^{-1}\right),
\end{align*}
for any $C>0$. We use apply this estimate for $g(x)$ so that
\begin{align*}\label{trunc}
&2i g\left(x\right)\\
=&\int_{-\infty}^\infty \cos\left(x \cosh\left(\pi y\right)\right)\left(Ge\left(Ty\right)\left(h\left(u\right)\left(Gu+T\right)\right)^\wedge\left(Gy\right)+Ge\left(-Ty\right)\left(h\left(u\right)\left(Gu+T\right)\right)^\wedge\left(-Gy\right) \right)dy\\
&+O_{B,C}\left(T^{-B}+T^{-C}\right)\\
=&2\int_{-\infty}^\infty \cos\left(x \cosh\left(\frac{\pi y}{G}\right)\right)\left(e\left(\frac{Ty}{G}\right)\left(h\left(u\right)\left(Gu+T\right)\right)^\wedge(y)\right)dy+O_{B,C}\left(T^{-B}+T^{-C}\right)\\
=& \sum_{\pm}\int_{-\infty}^\infty \exp\left(\frac{2\pi i T y}{G}\pm ix \cosh\left(\frac{\pi y}{G}\right)\right)\left(h\left(u\right)\left(Gu+T\right)\right)^\wedge\left(y\right) dy+O_{B,C}\left(T^{-B}+T^{-C}\right),
\end{align*}
for any $B>0$ and $C>0$.

Since both $\left(h\left(u\right)\right)^\wedge\left(y\right)$ and $\left(h\left(u\right)u\right)^\wedge\left(y\right)$ are rapidly decaying, there exist compactly supported smooth functions $h_1$ and $h_2$ whose supports are in $(-T^{\epsilon/2},T^{\epsilon/2})$ such that
\[
\frac{\partial^k}{\partial y^k} h_j (y) \ll_{k} T^{-k\epsilon/2}
\]
for all $k\geq 0$ and $j=1,2$, and that
\begin{multline*}
\sum_{\pm}\int_{-\infty}^\infty \exp\left(\frac{2\pi i T y}{G}\pm ix \cosh\left(\frac{\pi y}{G}\right)\right)\left(h\left(u\right)\left(Gu+T\right)\right)^\wedge\left(y\right) dy\\
= \sum_{\pm}\int_{-\infty}^\infty \exp\left(\frac{2\pi i T y}{G}\pm ix \cosh\left(\frac{\pi y}{G}\right)\right)\left(Gh_1(y)+Th_2(y)\right) dy+O_{\epsilon, D} (T^{-D}).
\end{multline*}
For $y \in (-T^{\epsilon/2},T^{\epsilon/2})$, because $x < KGT^{1-\epsilon}$,
\begin{align*}
\frac{\partial}{\partial y}\left(\frac{2\pi T y}{G}\pm x \cosh\left(\frac{\pi y}{G}\right)\right) &= \frac{2\pi  T}{G}+ x\frac{\pi}{G}\sinh\left(\frac{\pi y}{G}\right)\\
&> \frac{2\pi  T}{G}- \frac{K\pi T^{1-\epsilon/2}}{G}\\
&> \frac{T}{G}
\end{align*}
for all sufficiently large $T$. We also have for $k \geq 2$ that
\begin{align*}
\frac{\partial^k}{\partial y^k}\left(\frac{2\pi  T y}{G}\pm x \cosh\left(\frac{\pi y}{G}\right)\right) &= \pm
\left\{\begin{array}{cl} \frac{x\pi^k}{G^k}\sinh \frac{\pi y}{G}&\hspace{10mm} k\text{: odd}\\ \frac{x\pi^k}{G^k}\cosh \frac{\pi y}{G} & \hspace{10mm} k \text{: even}
\end{array} \right.\\
&\ll_k T^{1-\epsilon/2}G^{-k} \ll TG^{-k}.
\end{align*}
Now define $I_n^j(y)$ by $I_0^j(y) = h_j(y)$ and
\[
I_n^j(y)= -\left(\frac{1}{\Phi_\pm(y)'} I_{n-1}^j(y)\right)'
\]
where $\Phi_\pm(y) = \frac{2\pi T y}{G}\pm x \cosh\frac{\pi y}{G}$, so that
\[
\int_{-\infty}^\infty  \exp\left(\frac{2\pi i T y}{G}\pm ix \cosh\left(\frac{\pi y}{G}\right)\right) h_j(y)dy=\int_{-\infty}^\infty  \exp\left(\frac{2\pi i T y}{G}\pm ix \cosh\left(\frac{\pi y}{G}\right)\right) I_n^j dy.
\]
Then $I_n^j(y)$ is a linear combination of
\[
h_j^{(a_0)}(y)\prod_{l=1}^n \frac{\partial^{a_l}}{\partial y^{a_l}}\frac{1}{\Phi_\pm (y)'}
\]
where $\sum_{l=0}^n a_l = n$. For each $l$, $\frac{\partial^{a_l}}{\partial y^{a_l}}\frac{1}{\Phi_\pm (y)'}$ is a linear combination of
\[
\frac{1}{(\Phi_\pm (y)')^{b_l+1}} \prod_{m=1}^{b_l} \frac{\partial^{b_{lm}}}{\partial y^{b_{lm}}} \Phi_\pm (y)'
\]
where $\sum_{m=1}^l b_{lm}=a_l$ and $b_l \leq a_l$, and $b_{lm}\geq 1$. Therefore
\begin{align*}
\frac{1}{(\Phi_\pm (y)')^{b_l+1}} \prod_{m=1}^{b_l} \frac{\partial^{b_{lm}}}{\partial y^{b_{lm}}} \Phi_\pm (y)' &\ll_n T^{-b_l-1}G^{b_l+1} \prod_{m=1}^{b_l} T G^{-b_{lm}-1}\\
&=T^{-1}G^{1-a_l},
\end{align*}
and
\begin{align*}
h_j^{(a_0)}(y)\prod_{l=1}^n \frac{\partial^{a_l}}{\partial y^{a_l}}\frac{1}{\Phi_\pm (y)'}\ll_n T^{-a_0\epsilon/2}T^{-n}G^{a_0}\ll G^{n}T^{-n-n\epsilon/2}.
\end{align*}
From this we deduce that
\[
\int_{-\infty}^\infty  \exp\left(\frac{2\pi i T y}{G}\pm ix \cosh\left(\frac{\pi y}{G}\right)\right) h_j(y)dy= \ll_n T^{\epsilon/2} G^n T^{-n-n\epsilon/2},
\]
and therefore combining all the estimates we conclude that $g(x) = O_{ \epsilon,A,K } (T^{-A})$
\end{proof}

Now let
\[
\tilde{g}\left(x\right)=\int_0^\infty J_{2iy}\left(x\right) \frac{h_{T,G}\left(y\right)y}{\cosh \pi y}dy.
\]
Note that
\[
\tilde{g}(x)-\overline{\tilde{g}(x)} = \int_0^\infty \left(J_{2iy}\left(x\right)-J_{-2iy}(x)\right) \frac{h_{T,G}\left(y\right)y}{\cosh \pi y}dy=g(x),
\]
hence $g(x)$ is the imaginary part of $2\tilde{g}(x)$.
\begin{lemma}\label{lem2}
Assume that $GT^{1-\epsilon}<x$ with $0< \epsilon <\theta/2$. For any $A>0$, there exists $N>0$ such that, $\tilde{g}\left(x\right)$ is a linear combination of
\[
\int_{|y-T|<T/2}g_{k,\kappa,N}\left(y,x\right) y h_{T,G}\left(y\right) dy \tag{$k=0,1,\ldots,N$, $\kappa=0,1,\ldots, k$}
\]
plus $O\left(T^{-A}\right)$, where
\[
g_{k,\kappa,N}\left(y,x\right)=y^{2\kappa}\left(4y^2+x^2\right)^{-k/2-\kappa-1/4}\exp\left(ix+i\sum_{m=1}^{N-1}c_m\frac{y^{2m}}{x^{2m-1}}\right)
\]
with some explicit constants $c_m$ (here $c_1=-2 \neq 0$).
\end{lemma}
\begin{proof}
Since $J_{2iy}(x)/\cosh \pi y$ is uniformly bounded,
\[
\int_{|y-T|\geq T/2} J_{2iy}\left(x\right) \frac{h_{T,G}\left(y\right)y}{\cosh \pi y}dy \ll \int_{|y-T|\geq T/2} h_{T,G}\left(y\right)ydy = O_B(T^{-B})
\]
for any $B>0$.

Now for $T/2<y<3T/2$ and $x>GT^{1-\epsilon}> T^{1+\theta/2}$, we infer from Proposition \ref{prop:asymp} that for any $N>0$, $J_{2iy}(x)/\cosh \pi y$ is a linear combination of
\[
y^{2\kappa}\left(4y^2+x^2\right)^{-k/2-\kappa -1/4}\exp\left(i\sqrt{4y^2+x^2}-2iy \sinh^{-1}\left(2y/x\right)\right)
\]
with $k=0,1,\ldots,N$ and $\kappa=0,1,\ldots, k$, plus an error term bounded from above by $\ll x^{-N-1} \ll T^{-N}$.

We now expand the exponent so that
\[
\exp\left(ix\sqrt{1+4y^2/x^2}-2iy \sinh^{-1}\left(2y/x\right)\right)
=\exp\left(ix+i\sum_{m=1}^{N-1}c_m\frac{y^{2m}}{x^{2m-1}}\right)+O\left(y^{2N}x^{-2N+1}\right).
\]
with some explicit constants $c_m$. Note that $y^{2N}x^{-2N+1}\ll_N T^{2N} T^{-(2N-1)(1+\theta/2)}= T^{-N\theta + 1+\theta}\ll T^{2-N\theta}$. From this we infer that for $T/2<y<3T/2$, $J_{2iy}(x)$ is a linear combination of $g_{k,\kappa,N}$ ($k=0,1,\ldots,N$ and $\kappa=0,1,\ldots, k$) plus an error term which is bounded from above by $O_N\left(T^{2- N\theta}\right)$. Therefore the theorem follows by taking $N=(A+100)/\theta$.
\end{proof}
\begin{lemma}\label{lem3}
For $0<x<1$,
\[
g\left(x\right) \ll Gx^{7/8}.
\]
\end{lemma}
\begin{proof}
Firstly, note that from \eqref{bes:series} that for $0<x<1$ and $\nu\in \mathbb{C}$ with $\mathrm{Re}(\nu)\geq 0$,
\begin{align*}
\left|\Gamma(\nu+1)J_{\nu}(x)\right| &= \left|(x/2)^\nu \sum_{k=0}^\infty (-1)^k \frac{(x/2)^{2k}}{k!(\nu+1)(\nu+2)\ldots(\nu+k)}\right|\\
&\leq (x/2)^{\mathrm{Re}(\nu)}\sum_{k=0}^\infty (-1)^k \frac{(x/2)^{2k}}{k!k!}\\
&\ll x^{\mathrm{Re}(\nu)}.
\end{align*}
Now from Stirling's formula, assuming that $0\leq \sigma <3$, we obtain
\[
J_{\sigma+2iy} (x) \ll \frac{x^\sigma}{|\Gamma(\sigma+1+2iy)|}\ll x^{\sigma}  (1+|y|)^{-\sigma-1/2} \cosh \pi y.
\]
uniformly in $0<x<1$ and $y\in \mathbb{R}$.

Now consider
\[
F(z):= J_{2iz}(x) \frac{h_{T,G}(z)z}{\cosh \pi z}.
\]
Because of the denominator $\cosh \pi z$, $F(z)$ has simple poles at $z=i\left(k+\frac{1}{2}\right)$ with $k\in \mathbb{Z}$. If $\mathrm{Re}(z)$ is fixed, $|F(z)| \ll_{T,G} \exp (-|\mathrm{Im}(z)|^2/G)$ as $|\mathrm{Im}(z)| \to \infty$. Therefore by shifting the contour from $\gamma_1=(-\infty,+\infty)$ to $\gamma_2=-\frac{7}{16}i+(-\infty,+\infty)$, we see that
\begin{align*}
g(x)&= \int_{\gamma_1} F(z)dz \\
&=\int_{\gamma_2} F(z)dz\\
&=\int_{-\infty}^\infty J_{7/8+2iy}(x) \frac{(-\frac{7}{16}i+y)h_{T,G}(-\frac{7}{16}i+y)}{\cosh \pi (-\frac{7}{16}i+y)}dy\\
&\ll x^{7/8} \int_{-\infty}^\infty |h_{T,G} (-\frac{7}{16}i+y)|dy\\
&\leq x^{7/8}\int_{-\infty}^\infty |\exp(-(-\frac{7}{16}i+y-T)^2/G^2)|+|\exp(-(-\frac{7}{16}i+y+T)^2/G^2)|dy\\
&=2x^{7/8}\int_{-\infty}^\infty |\exp(-(-\frac{7}{16}i+y)^2/G^2)|dy\\
&=2x^{7/8}\int_{-\infty}^\infty \exp((y^2-\frac{7^2}{16^2})/G^2)dy\\
&\ll x^{7/8} G. \qedhere
\end{align*}
\end{proof}
\begin{remark}
In Lemma \ref{lem1} and \ref{lem2}, the fact that $h\left(y\right)$ is a rapidly decreasing function is needed only, but for Lemma \ref{lem3}, analyticity is required.
\end{remark}

\section{Quantitative quantum ergodicity-I}\label{qqe1}
In this section, we prove Theorem \ref{thm} for positive $m$. The proof when $m$ is negative follows identically. For the rest of the section, we assume that $G=T^\theta$ with $1/3<\theta<1$.
\subsection{Reduction via Kuznetsov trace formula}


We first use Kuznetsov trace formula with the test function $h_{T,G}(y)$ from Lemma \ref{kuz} to transform the sum of shifted convolution sums into an oscillating exponential sum.
\begin{lemma}\label{lem:kuz}
Let $X$ be a parameter which may vary with $T$ with the constraint $1 \ll X \ll T$. Let $\psi \in C_0^\infty (0,\infty)$ be a real-valued function. Assume that $m$ is an integer in the range $0<m<X$. Let $I^{diag}$ and $I^{OD}$ be given by
\[
I^{diag} =\sum_{d|m}\sum_{r_1,r_2} \delta_{r_1\left(r_1+d\right),r_2\left(r_2+d\right)} \psi\left(\frac{\pi mr_1}{dX}\right)\psi\left(\frac{\pi mr_2}{dX}\right)
\int_\mathbb{R} \tanh\left(\pi y\right) h_{T,G}\left(y\right) y dy
\]
and
\[
I^{OD}=\sum_{d|m}\sum_{c=1}^\infty \sum_{r_1,r_2}\psi\left(\frac{\pi mr_1}{dX}\right)\psi\left(\frac{\pi mr_2}{dX}\right) \frac{S\left(r_1\left(r_1+d\right),r_2\left(r_2+d\right),c\right)}{c} g\left(\frac{4\pi}{c}\sqrt{r_1\left(r_1+d\right)r_2\left(r_2+d\right)}\right),
\]
where $g$ is defined by \eqref{inttrans}. Then for any fixed $\epsilon>0$, we have
\[
\sum_{|t_\phi-T|<G}\left| \sum_n \rho_\phi\left(n+m\right)\rho_\phi\left(n\right) \psi\left(\frac{\pi n}{X}\right)\right|^2 \ll_\epsilon T^\epsilon\left(|I^{diag}|+|I^{OD}|\right).
\]
\end{lemma}
\begin{proof}
We first use \eqref{triv} so that
\begin{align*}
\sum_{|t_\phi-T|<G}\left| \sum_n \rho_\phi\left(n+m\right)\rho_\phi\left(n\right) \psi\left(\frac{\pi n}{X}\right)\right|^2 &\ll_\epsilon T^\epsilon \sum_{|t_\phi-T|<G}\frac{1}{\rho_\phi(1)^2}\left| \sum_n \rho_\phi\left(n+m\right)\rho_\phi\left(n\right) \psi\left(\frac{\pi n}{X}\right)\right|^2\\
&\ll  T^\epsilon \sum_{\phi}\frac{h_{T,G}\left(t_\phi\right)}{\rho_\phi\left(1\right)^2}\left|\sum_n \rho_\phi\left(n+m\right)\rho_\phi\left(n\right) \psi\left(\frac{\pi n}{X}\right)\right|^2.
\end{align*}
We rearrange the inner sum using the Hecke relation \eqref{heck} and then we apply Cauchy--Schwartz inequality as follows:
\begin{align*}
\sum_{\phi}\frac{h_{T,G}\left(t_\phi\right)}{\rho_\phi\left(1\right)^2}\left|\sum_n \rho_\phi\left(n+m\right)\rho_\phi\left(n\right) \psi\left(\frac{\pi n}{X}\right)\right|^2
&=\sum_{\phi}h_{T,G}\left(t_\phi\right)\left|\sum_{d|m}\sum_n \rho_\phi\left(n\left(n+d\right)\right) \psi\left(\frac{\pi m n}{dX}\right)\right|^2\\
&\leq \tau\left(m\right)\sum_{d|m}\sum_{\phi}h_{T,G}\left(t_\phi\right)\left|\sum_n \rho_\phi\left(n\left(n+d\right)\right) \psi\left(\frac{\pi m n}{dX}\right)\right|^2\\
&\ll_\epsilon  T^\epsilon \sum_{d|m}\sum_{\phi}h_{T,G}\left(t_\phi\right)\left|\sum_n \rho_\phi\left(n\left(n+d\right)\right) \psi\left(\frac{\pi m n}{dX}\right)\right|^2.
\end{align*}
Here we used $\tau(m) \ll_\epsilon m^\epsilon$ (Remark \ref{rem1}) and $m<X \ll T$.

Now we apply Kuznetsov trace formula \eqref{kuz} to obtain the following identity:
\begin{multline*}
\sum_{d|m}\sum_{\phi}h_{T,G}\left(t_\phi\right)\left|\sum_n \rho_\phi\left(n\left(n+d\right)\right) \psi\left(\frac{\pi m n}{dX}\right)\right|^2 \\
+\frac{1}{4\pi}\sum_{d|m} \int_{-\infty}^\infty \left|\sum_n \tau\left(n\left(n+d\right),r\right) \psi\left(\frac{\pi m n}{dX}\right)\right|^2 \frac{h_{T,G}(r)dr}{\cosh \pi r}
\end{multline*}
\begin{align*}
=&\frac{1}{\pi^2}\sum_{d|m}\sum_{r_1,r_2} \delta_{r_1\left(r_1+d\right),r_2\left(r_2+d\right)} \psi\left(\frac{\pi mr_1}{dX}\right)\psi\left(\frac{\pi mr_2}{dX}\right)
\int_\mathbb{R} \tanh\left(\pi y\right) h_{T,G}\left(y\right) y dy\\
+&\frac{2i}{\pi}\sum_{d|m}\sum_{c=1}^\infty \sum_{r_1,r_2}\psi\left(\frac{\pi mr_1}{dX}\right)\psi\left(\frac{\pi mr_2}{dX}\right) \frac{S\left(r_1\left(r_1+d\right),r_2\left(r_2+d\right),c\right)}{c} g\left(\frac{4\pi}{c}\sqrt{r_1\left(r_1+d\right)r_2\left(r_2+d\right)}\right)\\
=&\frac{1}{\pi^2}I^{diag}+\frac{2i}{\pi}I^{OD}.
\end{align*}
Note that we used the fact that $\tau(n,r)$ is real-valued and the assumption that $\psi$ is real.
\end{proof}

To prove Theorem \ref{thm}, we bound the diagonal contribution $I^{diag}$ and the off-diagonal contribution $I^{OD}$ separately.
\begin{lemma}\label{lem:diag}
Assume that $\psi\in C_0^\infty (0,\infty)$ is supported in $(1/l,l)$ for some $l>0$. Then for any $\epsilon>0$, we have
\[
I^{diag} \ll_{l, \epsilon} XGT^{1+\epsilon}\|\psi\|_{L^\infty}^2.
\]
\end{lemma}
\begin{proof}
Recall that $h_{T,G}$ is defined by
\[
h_{T,G}\left(y\right)=h\left(\left(y-T\right)/G\right)+h\left(\left(y+T\right)/G\right)
\]
where $h\left(y\right)=e^{-y^2}$. We therefore bound the integral as follows:
\begin{align*}
&\int_{-\infty}^\infty \tanh\left(\pi y\right) h_{T,G}\left(y\right) y dy \\
=& G\int_{-\infty}^\infty (\tanh(\pi(Gx+T))(Gx+T)+\tanh(\pi(Gx-T))(Gx-T))h(x) dx\\
\leq & 2GT\int_{-\infty}^\infty h(x)dx+2G^2 \int_{-\infty}^\infty h(x)|x|dx\\
\ll &GT.
\end{align*}
To handle the summation over $r_1$ and $r_2$, note that $r_1(r_1+d) = r_2(r_2+d)$ if and only if $r_1=r_2$ or $r_1+r_2+d=0$. Since $\psi$ is assumed to be supported on the positive real line, we infer that the sum is equal to
\begin{align*}
\sum_{d|m} \sum_{r_1} \left|\psi\left(\frac{\pi mr_1}{dX}\right)\right|^2.
\end{align*}
We bound this under the assumption that $\psi$ is supported in a fixed interval $(1/l,l)$ as follows:
\begin{align*}
\sum_{d|m} \sum_{r_1}\left| \psi\left(\frac{\pi mr_1}{dX}\right)\right|^2 &\leq \sum_{d|m} \frac{dlX}{\pi m} \|\psi\|_{L^\infty}^2\\
&< \tau(m) l X \|\psi\|_{L^\infty}^2\\
&\ll_\epsilon l X T^\epsilon \|\psi\|_{L^\infty}^2.
\end{align*}
Here we used $\tau(m) \ll_\epsilon m^\epsilon$ (Remark \ref{rem1}) and $m<X^{\frac{1}{2}}\ll T$. We complete the proof by combining these two estimates.
\end{proof}
To handle the off-diagonal contribution $I^{OD}$, we claim the following.
\begin{lemma}\label{mainlemma}
Fix $l>0$ and $\epsilon>0$. Let $\psi \in C_0^\infty\left(1/l,l\right)$, $1\ll R \ll T$ and $0< d < R^{\frac{1}{2}}$. Then there exists $A>0$ depending only on $\theta$ and $\epsilon$ such that
\begin{multline}\label{exp}
 I^{OD'}:=\sum_{c\geq 1}\sum_{r_1,r_2} \psi\left(\frac{r_1}{R}\right)\psi\left(\frac{r_2}{R}\right)\frac{S\left(r_1\left(r_1+d\right),r_2\left(r_2+d\right),c\right)}{c}  g\left(\frac{4\pi \sqrt{r_1r_2\left(r_1+d\right)\left(r_2+d\right)}}{c}\right)\\
\ll_{\epsilon,l,\theta} d^{3/2}GRT^{1+\epsilon}\|\psi\|_{W^{A,\infty}}^2.
\end{multline}
\end{lemma}
Substituting $\frac{dX}{m}$ by $R=R(d,m,X)$, we find that $R$ and $d$ satisfy $1 \ll R \leq X \ll T$, and $0< d < R^{\frac{1}{2}}$. From this, substituting $\psi(\pi x)$ by $\psi_0(x)$, we bound $I^{OD}$ using Lemma \ref{mainlemma} as follows
\begin{align*}
I^{OD}&=\sum_{d|m}\sum_{c=1}^\infty \sum_{r_1,r_2}\psi\left(\frac{\pi mr_1}{dX}\right)\psi\left(\frac{\pi mr_2}{dX}\right) \frac{S\left(r_1\left(r_1+d\right),r_2\left(r_2+d\right),c\right)}{c} g\left(\frac{4\pi}{c}\sqrt{r_1\left(r_1+d\right)r_2\left(r_2+d\right)}\right)\\
&=\sum_{d|m}\sum_{c=1}^\infty \sum_{r_1,r_2}\psi_0\left(\frac{r_1}{R}\right)\psi_0\left(\frac{r_2}{R}\right) \frac{S\left(r_1\left(r_1+d\right),r_2\left(r_2+d\right),c\right)}{c} g\left(\frac{4\pi}{c}\sqrt{r_1\left(r_1+d\right)r_2\left(r_2+d\right)}\right)\\
&\ll_{l,\theta, \epsilon} \sum_{d|m} d^{3/2}\frac{dX}{m} GT^{1+\epsilon}\|\psi\|_{W^{A,\infty}}^2\\
&\leq \tau(m)m^{3/2}XGT^{1+\epsilon}\|\psi\|_{W^{A,\infty}}^2\\
&\ll m^{3/2}XGT^{1+2\epsilon}\|\psi\|_{W^{A,\infty}}^2,
\end{align*}
where we used Remark \ref{rem1}.  We therefore conclude that Lemma \ref{lem:kuz}, \ref{lem:diag}, and \ref{mainlemma} implies Theorem \ref{thm}. For the rest of the section, we prove Lemma \ref{mainlemma}.




\subsection{Off-diagonal contribution}
\subsubsection{Estimating the tail and further reduction}
We begin by estimating the tail of the sum over $c$ in $I^{OD'}$.
\begin{lemma}\label{tail}
For any fixed $\epsilon>0$, we have
\begin{multline*}
I_{tail}^{OD'}=\sum_{c\geq R^2G^{-1}T^{-1+\epsilon}}\sum_{r_1,r_2} \psi\left(\frac{r_1}{R}\right)\psi\left(\frac{r_2}{R}\right)\frac{S\left(r_1\left(r_1+d\right),r_2\left(r_2+d\right),c\right)}{c}  g\left(\frac{4\pi \sqrt{r_1r_2\left(r_1+d\right)\left(r_2+d\right)}}{c}\right)\\
=O_{\epsilon,l} \left(T^{-100}\|\psi\|_{L^\infty}^2\right).
\end{multline*}
\end{lemma}
\begin{proof}
Since the support of $\psi$ is contained in $(1/l,l)$, we have $R/l<r_1, r_2 < lR$. Therefore for $c>100l^2R^2$, we have from Lemma \ref{lem3} that
\[
g\left(\frac{4\pi \sqrt{r_1r_2\left(r_1+d\right)\left(r_2+d\right)}}{c}\right) \ll G\frac{l^{7/4}R^{7/4}}{c^{7/8}}.
\]
Also, Lemma \ref{weilb} implies that
\[
S\left(r_1\left(r_1+d\right),r_2\left(r_2+d\right),c\right) \ll lRc^{5/8}
\]
where we used $\tau(c) \ll c^{1/8}$ and $\gcd(n,m,c) \leq n$. Hence for any parameter $Y>100l^2R^2$, we have
\begin{align*}
\sum_{c\geq Y}\sum_{r_1,r_2} & \psi\left(\frac{r_1}{R}\right)\psi\left(\frac{r_2}{R}\right)\frac{S\left(r_1\left(r_1+d\right),r_2\left(r_2+d\right),c\right)}{c}  g\left(\frac{4\pi \sqrt{r_1r_2\left(r_1+d\right)\left(r_2+d\right)}}{c}\right)\\
&\ll \sum_{c\geq Y}\sum_{r_1,r_2} \psi\left(\frac{r_1}{R}\right)\psi\left(\frac{r_2}{R}\right) \frac{Gl^{11/4}R^{11/4}}{c^{5/4}}\\
&\leq \sum_{c\geq Y} \frac{Gl^{11/4}R^{11/4}}{c^{5/4}} \|\psi\|_{L^\infty}^2\\
&\ll Gl^{11/4}R^{11/4}Y^{-1/4} \|\psi\|_{L^\infty}^2.
\end{align*}
Now by Lemma \ref{lem1}, for any $A>0$, we have
\begin{align*}
\sum_{R^2G^{-1}T^{-1+\epsilon}\leq c< Y}\sum_{r_1,r_2} & \psi\left(\frac{r_1}{R}\right)\psi\left(\frac{r_2}{R}\right)\frac{S\left(r_1\left(r_1+d\right),r_2\left(r_2+d\right),c\right)}{c}  g\left(\frac{4\pi \sqrt{r_1r_2\left(r_1+d\right)\left(r_2+d\right)}}{c}\right)\\
&\ll_{A,\epsilon,l} \sum_{R^2G^{-1}T^{-1+\epsilon}\leq c< Y}\sum_{r_1,r_2} \psi\left(\frac{r_1}{R}\right)\psi\left(\frac{r_2}{R}\right) lRT^{-A}\\
&\leq Yl^3 R^3  T^{-A}\|\psi\|_{L^\infty}^2,
\end{align*}
for all sufficiently large $T$. We complete the proof by choosing $Y=T^{800}$ and $A=1600$.
\end{proof}



From this lemma, we may assume that $GT^{1-\epsilon_1} < R^2$ with some fixed constant $\epsilon_1$ that satisfies
\[
\min\left\{\frac{\theta}{2},\frac{3\theta-1}{6}\right\}>\epsilon_1>0.
\]
Observe that $g\left(x\right)$ is the imaginary part of $2\tilde{g}\left(x\right)$, that $\psi$ is real valued, and that the Kloosterman sums are real. Therefore
\begin{multline*}
I_{main}^{OD'}=\left|\sum_{c< R^2G^{-1}T^{-1+\epsilon_1}}\sum_{r_1,r_2} \psi\left(\frac{r_1}{R}\right)\psi\left(\frac{r_2}{R}\right)\frac{S\left(r_1\left(r_1+d\right),r_2\left(r_2+d\right),c\right)}{c}  g\left(\frac{4\pi \sqrt{r_1r_2\left(r_1+d\right)\left(r_2+d\right)}}{c}\right)\right|\\
\leq \left|\sum_{c< R^2G^{-1}T^{-1+\epsilon_1}}\sum_{r_1,r_2} \psi\left(\frac{r_1}{R}\right)\psi\left(\frac{r_2}{R}\right)\frac{S\left(r_1\left(r_1+d\right),r_2\left(r_2+d\right),c\right)}{c}  2\tilde{g}\left(\frac{4\pi \sqrt{r_1r_2\left(r_1+d\right)\left(r_2+d\right)}}{c}\right)\right|.
\end{multline*}

From Lemma \ref{lem2} with $A=200$ and sufficiently large $N>0$, we see that
\[
\sum_{c< R^2G^{-1}T^{-1+\epsilon_1}}\sum_{r_1,r_2} \psi\left(\frac{r_1}{R}\right)\psi\left(\frac{r_2}{R}\right)\frac{S\left(r_1\left(r_1+d\right),r_2\left(r_2+d\right),c\right)}{c}  \tilde{g}\left(\frac{4\pi \sqrt{r_1r_2\left(r_1+d\right)\left(r_2+d\right)}}{c}\right)
\]
is a linear combination of
\begin{multline*}
\sum_{c< R^2G^{-1}T^{-1+\epsilon_1}}\sum_{r_1,r_2} \psi\left(\frac{r_1}{R}\right)\psi\left(\frac{r_2}{R}\right)\frac{S\left(r_1\left(r_1+d\right),r_2\left(r_2+d\right),c\right)}{c}  \\
\times \int_{|y-T|<T/2} g_{k,\kappa,N}\left(y,\frac{4\pi \sqrt{r_1r_2\left(r_1+d\right)\left(r_2+d\right)}}{c}\right)yh_{T,G}(y)dy\tag{$k=0,1,\ldots, N$},
\end{multline*}
plus an error term that contributes at most $O_l(T^{-100})$.
\begin{remark}
Since we are assuming that $\theta>1/3$, we can take, for instance, $N=900$.
\end{remark}
Therefore it is sufficient to prove that
\begin{lemma}\label{mainlem}
Assume that $|y-T|<T/2$ and that $0<d<R^{\frac{1}{2}}$. Then for any given sufficiently small $\epsilon_0>0$ and $A>\max\{100/(3\theta-1), 300/\epsilon_0\}$, we have
\begin{multline}\label{asdf}
J_{main}:=\sum_{c< R^2G^{-1}T^{-1+\epsilon_1}}\sum_{r_1,r_2} \psi\left(\frac{r_1}{R}\right)\psi\left(\frac{r_2}{R}\right)\frac{S\left(r_1\left(r_1+d\right),r_2\left(r_2+d\right),c\right)}{c}\\
\times g_{k,\kappa,N}\left(y,\frac{4\pi \sqrt{r_1r_2\left(r_1+d\right)\left(r_2+d\right)}}{c}\right)
 \ll_{l, \theta, \epsilon_0} d^{3/2} RT^{\epsilon_0} \|\psi\|_{W^{A,\infty}}.
\end{multline}
\end{lemma}
Indeed, if we know \eqref{asdf}, then
\begin{multline*}
\sum_{c< R^2G^{-1}T^{-1+\epsilon_1}}\sum_{r_1,r_2} \psi\left(\frac{r_1}{R}\right)\psi\left(\frac{r_2}{R}\right)\frac{S\left(r_1\left(r_1+d\right),r_2\left(r_2+d\right),c\right)}{c}  \\
\times \int_{|y-T|<T/2} g_{k,\kappa,N}\left(y,\frac{4\pi \sqrt{r_1r_2\left(r_1+d\right)\left(r_2+d\right)}}{c}\right)yh_{T,G}(y)dy\ll_{l,\theta,\epsilon_0} d^{3/2} RGT^{1+\epsilon_0} \|\psi\|_{W^{A,\infty}},
\end{multline*}
and therefore $I_{main}^{OD'} \ll_{l,\theta,\epsilon_0} d^{3/2} RGT^{1+\epsilon_0} \|\psi\|_{W^{A,\infty}}$. Combined with the bound
\[
I_{tail}^{OD'}= O_{\epsilon_1,l} \left(T^{-100}\|\psi\|_{L^\infty}^2\right)
\]
from Lemma \ref{tail}, and $|I^{OD'}|\leq |I_{main}^{OD'}|+|I_{tail}^{OD'}|$, we obtain Lemma \ref{mainlemma}.
\subsubsection{Splitting oscillating factors}
To prove Lemma \ref{mainlem}, we investigate the cancellation coming from the sum over $r_1$ and $r_2$ in \eqref{asdf}, as done in \cite{sarluo}. To this end, we first want to represent
\[
\psi\left(\frac{r_1}{R}\right)\psi\left(\frac{r_2}{R}\right)g_{k,\kappa,N}\left(y,\frac{4\pi \sqrt{r_1r_2\left(r_1+d\right)\left(r_2+d\right)}}{c}\right)
\]
as a product of the main oscillating factor times something that oscillates mildly. For $c<R^2G^{-1}T^{-1+\epsilon_1}$, $r_1, r_2 \sim R$, we have
\[
\frac{4\pi \sqrt{r_1r_2\left(r_1+d\right)\left(r_2+d\right)}}{c} \gg GT^{1-\epsilon_1}.
\]
Note that, when $x \gg GT^{1-\epsilon_1}$ and $y \sim T$, the main oscillating factor of
\[
g_{k,\kappa,N}\left(y,x\right)=y^{2\kappa}\left(4y^2+x^2\right)^{-k/2-\kappa-1/4}\exp\left(ix+i\sum_{m=1}^{N-1}c_m\frac{y^{2m}}{x^{2m-1}}\right)
\]
in $x$ aspect is $\exp (ix)$. Now for $r_1\sim R$, $r_2 \sim R$, and $d <R^{\frac{1}{2}}$, we have
\begin{align*}
\exp\left(\frac{4\pi i\sqrt{r_1r_2\left(r_1+d\right)\left(r_2+d\right)}}{c}\right)&=\exp\left(\frac{4\pi ir_1r_2\sqrt{\left(1+d/r_1\right)\left(1+d/r_2\right)}}{c}\right)\\
&=\exp\left(\frac{4\pi ir_1r_2(1+\frac{d}{2r_1}+O(d^2/R^2))(1+\frac{d}{2r_2}+O(d^2/R^2))}{c}\right)\\
&=\exp\left(\frac{2\pi i(2r_1r_2+dr_1+dr_2+O(d^2))}{c}\right),
\end{align*}
and from this calculation, we infer that the main oscillating factor of $\exp\left(4\pi i\sqrt{r_1r_2\left(r_1+d\right)\left(r_2+d\right)}/c\right)$ is given by $e_c\left(2r_1r_2+dr_1+dr_2\right)$.

Motivated from these observations, we define $f_c\left(r_1,r_2\right)$ by the following equation
\[
e_c\left(2r_1r_2+dr_1+dr_2\right)f_c\left(r_1,r_2\right)
=\psi \left(\frac{r_1}{R}\right)\psi \left(\frac{r_2}{R}\right)g_{k,\kappa,N}\left(y,\frac{4\pi \sqrt{r_1r_2\left(r_1+d\right)\left(r_2+d\right)}}{c}\right).
\]
Now, for each fixed $c$, we rearrange the sum in \eqref{asdf} modulo $c$ as follows
\begin{align*}
cJ_{main}(c)=&\sum_{r_1,r_2} S\left(r_1\left(r_1+d\right),r_2\left(r_2+d\right),c\right) e_c\left(2r_1r_2+dr_1+dr_2\right)f_c\left(r_1,r_2\right)\\
=& \sum_{a,b\pmod{c}}S\left(a\left(a+d\right),b\left(b+d\right),c\right)e_c\left(2ab+da+db\right) \sum_{\substack{r_1 \equiv a\pmod{c} \\ r_2 \equiv b\pmod{c} }}f_c\left(r_1,r_2\right)\\
=&\frac{1}{c^2} \sum_{u\pmod{c}} \sum_{v\pmod{c}}\left(\sum_{a,b\pmod{c}} S\left(a\left(a+d\right),b\left(b+d\right),c\right)e_c\left(2ab+\left(d+u\right)a+\left(d+v\right)b\right) \right)\\
&\times \sum_{r_1,r_2}f_c\left(r_1,r_2\right)e_c\left(-ur_1-vr_2\right).
\end{align*}
(Note that $\sum_{c<R^2G^{-1}T^{-1+\epsilon_1}}J_{main}(c) = J_{main}$ in \eqref{asdf}.) Assume without loss of generality that $|u|,|v| \leq \frac{c}{2}$. As we expect $f_c\left(r_1,r_2\right)$ is mildly oscillating, the sum
\[
\sum_{r_1,r_2}f_c\left(r_1,r_2\right)e_c\left(-ur_1-vr_2\right)
\]
is going to be negligible, unless both $u$ and $v$ are relatively smaller than $c$. We quantify this using the Poisson summation formula and prove the following estimate in \S\ref{section212}.
\begin{lemma}\label{lem:poisson}
Assume that $|y-T|<T/2$, and $0<d<R^{\frac{1}{2}}$. Fix a constant $\epsilon_2>0$, and assume that $dR^{\epsilon_2}<R^2G^{-1}T^{-1+\epsilon_1}$, and that $dR^{\epsilon_2}<c<R^2G^{-1}T^{-1+\epsilon_1}$. Let $A>0$ be a fixed positive integer such that $A>\max\{100/(3\theta-1),100/\epsilon_2\}$. Then we have
\[
\sum_{r_1,r_2}f_c\left(r_1,r_2\right)e_c\left(-ur_1-vr_2\right) = O_{l, \theta, \epsilon_2} \left(c^{-1/2} R^3  T^{-2} \log T \|\psi\|_{W^{A,\infty}}^2\right).
\]
Also, there exists a constant $\eta>0$ depending only on $l$ such that if either $u$ or $v$ is not in the range
\[
\lbrack \eta^{-1}\frac{c^2T^2}{2\pi R^3},\eta\frac{c^2T^2}{2\pi R^3}\rbrack
\]
then
\[
\sum_{r_1,r_2}f_c\left(r_1,r_2\right)e_c\left(-ur_1-vr_2\right)= O_{l, \theta, \epsilon_2}\left( T^{-20}\|\psi\|_{W^{A,\infty}\left(0,\infty\right)}^2\right).
\]
\end{lemma}
For the sums of Kloosterman sums, we know the following estimate from \cite{sarluo}.
\begin{lemma}\label{lem:kloos}
Let $c=c_1c_2$ with $\left(2,c_1\right)=1$ and $c_2|2^\infty$. Then for any fixed $\epsilon>0$, we have that
\begin{multline*}
\sum_{a,b\pmod{c}} S\left(a\left(a+d\right),b\left(b+d\right),c\right)e_c\left(2ab+\left(d+u\right)a+\left(d+v\right)b\right)\\
= \left\{
\begin{array}{cl}O_\epsilon\left(\left(v,c_1\right)c_1^{3/2}c_2^{5/2+\epsilon}\right) &\hspace{10mm} \text{ if }\left(u,c_1\right)=\left(v,c_1\right),\\0 & \hspace{10mm}\text{ otherwise }
\end{array} \right.
\end{multline*}
\end{lemma}
\begin{remark}
In \cite{sarluo}, the first condition is given by $(v,c_1) | u$ instead of $(v,c_1)=(u,c_1)$. These two conditions are in fact the same, because the sum vanishes unless both $(v,c_1)|u$ and $(u,c_1)|v$  hold, by the symmetry.
\end{remark}
\subsubsection{Completion of proof of Lemma \ref{mainlem}}
Before we prove Lemma \ref{lem:poisson}, we complete the proof of Lemma \ref{mainlem} using Lemma \ref{lem:poisson} and Lemma \ref{lem:kloos}.
\begin{lemma}
Fix $\epsilon>0$ and $\epsilon_2>0$. For $dR^{\epsilon_2}<c<R^2G^{-1}T^{-1+\epsilon_1}$ and $|y-T|<T/2$, we have
\[
J_{main}(c) \ll_{l,\theta, \epsilon_2,\epsilon}  c_1^{2}c_2^{3}R^{-3}T^{2+\epsilon} \|\psi\|_{W^{A,\infty}}^2
\]
for any $A>\max\{100/(3\theta-1),100/\epsilon_2\}$, where $c=c_1c_2$ with $\left(2,c_1\right)=1$ and $c_2|2^\infty$.
\end{lemma}
\begin{proof}
As a direct consequence of Lemma \ref{lem:poisson} and Lemma \ref{lem:kloos}, we have
\[
J_{main}(c) \ll_{l,\theta, \epsilon_2,\epsilon} c^{-3}\sum_{\substack{u,v \in \lbrack \eta^{-1}\frac{c^2T^2}{2\pi R^3},\eta\frac{c^2T^2}{2\pi R^3}\rbrack\\ \left(u,c_1\right)=\left(v,c_1\right)}} \left(v,c_1\right)c_1^{3/2}c_2^{5/2+\epsilon}c^{-1/2}R^3T^{-2}\log T \|\psi\|_{W^{A,\infty}}^2.
\]
Note that, for any given $M>0$, we have
\begin{align*}
\sum_{\substack{u,v < M \\ \left(u,c_1\right)=\left(v,c_1\right)}} (v,c_1) &= \sum_{\substack{d|c_1\\d<M}} d \sum_{\substack{u,v<M\\(u,c_1)=(v,c_1)=d}}1\\
&=\sum_{\substack{d|c_1\\d<M}} d \sum_{\substack{u,v<M/d\\(u,c_1/d)=(v,c_1/d)=1}}1\\
&\leq \sum_{\substack{d|c_1\\d<M}} d \sum_{\substack{u,v<M/d}}1\\
&\leq \sum_{\substack{d|c_1\\d<M}} M^2/d \\
&\leq \sum_{d<M}M^2/d \ll_\epsilon M^{2+\epsilon},
\end{align*}
hence
\[
J_{main}(c) \ll_{l,\theta, \epsilon_2,\epsilon} c_1^{3/2}c_2^{5/2+\epsilon}c^{1/2}R^{-3}T^{2}\log T \|\psi\|_{W^{A,\infty}}^2\ll_\epsilon c_1^{2}c_2^{3}R^{-3}T^{2+2\epsilon} \|\psi\|_{W^{A,\infty}}^2.\qedhere
\]
\end{proof}
For $c\leq dR^{\epsilon_2}$, we use Lemma \ref{weilb}, Remark \ref{rem1}, and Lemma \ref{lem:non}, so that
\begin{align*}
J_{main}(c) &\ll_l \sum_{r_1,r_2 \in (R/l,lR)} \frac{ (r_1(r_1+d),r_2(r_2+d),c)^{1/2}c^{1/2} \tau(c)}{c} c^{1/2}R^{-1}\|\psi\|_{L^\infty}^2\\
&\ll_{l,\epsilon} c^{1/2}RT^{\epsilon}\|\psi\|_{L^\infty}^2.
\end{align*}
We therefore have
\begin{align*}
J_{main}&= \sum_{c\leq dR^{\epsilon_2}} J_{main}(c)+\sum_{dR^{\epsilon_2}<c< R^2 G^{-1}T^{-1+\epsilon_1}} J_{main}(c)\\
&= O_{l,\epsilon} \left(d^{3/2}R^{1+3\epsilon_2/2}T^{\epsilon}\|\psi\|_{L^\infty}^2\right)+ O_{l,\theta,\epsilon_2,\epsilon}\left(R^{-3}T^{2+\epsilon} \|\psi\|_{W^{A,\infty}}^2\sum_{c< R^2 G^{-1}T^{-1+\epsilon_1}}c_1^{2}c_2^{3}\right),
\end{align*}
and
\begin{align*}
\sum_{c< R^2 G^{-1}T^{-1+\epsilon_1}}c_1^{2}c_2^{3}&<\sum_{\substack{c_2< R^2G^{-1}T^{-1+\epsilon_1}\\c_2|2^\infty}} c_2^3 \sum_{c_1< R^2G^{-1}T^{-1+\epsilon_1}/c_2}c_1^2\\
&\ll \sum_{\substack{c_2< R^2G^{-1}T^{-1+\epsilon_1}\\c_2|2^\infty}} R^6G^{-3}T^{-3+3\epsilon_1}\\
&\ll_\epsilon R^6G^{-3}T^{-3+3\epsilon_1+\epsilon}\\
&=R^6T^{-3+3\epsilon_1-3\theta+\epsilon}\\
&<R^6T^{-4-(3\theta-1)/2+\epsilon}\\
&\ll R^4T^{-2+\epsilon}
\end{align*}
where we used the assumption that $0<\epsilon_1 <(3\theta-1)/6$ and $R\ll T$. Combining these two estimates, we conclude that
\[
J_{main} \ll_{l,\theta,\epsilon_2, \epsilon} d^{3/2}RT^{\epsilon+\frac{3}{2}\epsilon_2}\|\psi\|_{W^{A,\infty}}^2.
\]
To obtain Lemma \ref{mainlem}, we choose $\epsilon_2 =  \epsilon_0 /3$ and $\epsilon=\epsilon_0/2$.
\subsection{Proof of Lemma \ref{lem:poisson}}\label{section212}
\subsubsection{Preliminary estimates}\label{section211}
Let
\begin{align*}
\Delta\left(r_1,r_2\right)&=\sqrt{r_1r_2\left(r_1+d\right)\left(r_2+d\right)}\\
\alpha\left(x\right)&=\sum_{m=1}^{N-1}c_m\frac{y^{2m}}{x^{2m-1}}\\
\varphi\left(r_1,r_2\right)&=\alpha\left(\frac{4\pi\Delta}{c}\right)+\frac{4\pi}{c}\left(\Delta-r_1r_2-\frac{dr_1}{2}-\frac{dr_2}{2}\right)
\end{align*}
and let
\[
g_c\left(r_1,r_2\right)=\psi\left(\frac{r_1}{R}\right)\psi\left(\frac{r_2}{R}\right)y^{2\kappa}\left(4y^2+\frac{16\pi^2r_1r_2\left(r_1+d\right)\left(r_2+d\right)}{c^2}\right)^{-k/2-\kappa-1/4}.
\]
Then, in these notations, we have
\[
f_c\left(r_1,r_2\right)=g_c\left(r_1,r_2\right) \exp\left(i\varphi\left(r_1,r_2\right)\right).
\]
Before we give a proof for Lemma \ref{lem:poisson}, we collect some estimates of derivatives of $f_c$, $g_c$, and $\varphi$.

\begin{lemma}\label{lem:non}
Assume $|y-T|<T/2$, $0<d<R^{\frac{1}{2}}$, $R^2>GT^{1-\epsilon_1}$, and $c<R^2G^{-1}T^{-1+\epsilon_1}$. For any nonnegative integers $k_1, k_2 \geq 0$, we have that
\begin{equation}\label{int2}
\frac{\partial^{k_1+k_2}g_c}{\partial {r_1}^{k_1}\partial {r_2}^{k_2}} \ll_{l,k_1,k_2} c^{1/2}R^{-1-k_1-k_2}\|\psi\|_{W^{k_1+k_2,\infty}(0,\infty)}^2.
\end{equation}
\end{lemma}
\begin{proof}
We may first assume without loss of generality that $r_1, r_2 \in ( R/l,lR)$, since $\psi$ is supported in $(1/l, l)$. By reexpressing $g_c$ as follows,
\[
g_c\left(r_1,r_2\right)=y^{2\kappa}2^{-2k-4\kappa-1}\pi^{-k-2\kappa-1/2}c^{k+2\kappa+\frac{1}{2}}\psi\left(\frac{r_1}{R}\right)\psi\left(\frac{r_2}{R}\right)\left(\frac{c^2y^2}{4\pi^2}+r_1r_2\left(r_1+d\right)\left(r_2+d\right)\right)^{-k/2-\kappa-1/4}.
\]
we see that
\[
\frac{\partial^{k_1+k_2}g_c}{\partial {r_1}^{k_1}\partial {r_2}^{k_2}}
\]
is a linear combination of
\begin{multline*}
y^{2\kappa}c^{k+2\kappa+\frac{1}{2}}R^{-n_1-n_2} \psi^{(n_1)}\left(\frac{r_1}{R}\right)\psi^{(n_2)}\left(\frac{r_2}{R}\right)\\
\times \left(\frac{c^2y^2}{4\pi^2}+r_1r_2\left(r_1+d\right)\left(r_2+d\right)\right)^{-k/2-\kappa-1/4-m}
\prod_{j=1}^m \frac{\partial^{a_j+b_j}}{\partial r_1^{a_j}\partial r_2^{b_j}}r_1r_2\left(r_1+d\right)\left(r_2+d\right)
\end{multline*}
where $n_1+\sum_{j=1}^m a_j = k_1$, $n_2+\sum_{j=1}^m b_j = k_2$, and $ 0\leq m \leq k_1+k_2-n_1-n_2$. We bound each term in product as follows:
\begin{align*}
|\psi^{(n_1)}\left(\frac{r_1}{R}\right)\psi^{(n_2)}\left(\frac{r_2}{R}\right) |&<  \|\psi\|_{W^{k_1+k_2,\infty}(0,\infty)}^2\\
\left(\frac{c^2y^2}{4\pi^2}+r_1r_2\left(r_1+d\right)\left(r_2+d\right)\right)^{-k/2-\kappa-1/4-m} &< \left(r_1r_2\left(r_1+d\right)\left(r_2+d\right)\right)^{-k/2-\kappa-1/4-m} \\
&\ll_l R^{-2k-4\kappa-1-4m}\\
\prod_{j=1}^m \frac{\partial^{a_j+b_j}}{\partial r_1^{a_j}\partial r_2^{b_j}}r_1r_2\left(r_1+d\right)\left(r_2+d\right)&\ll_l \prod_{j=1}^mR^{4- a_j-b_j}\\
&= R^{4m-k_1+n_1-k_2+n_2}.
\end{align*}
Combining these estimates, we infer that
\[
\frac{\partial^{k_1+k_2}g_c}{\partial {r_1}^{k_1}\partial {r_2}^{k_2}}\ll_{l,k_1,k_2} T^{2\kappa}c^{k+2\kappa+ 1/2}R^{-2k-4\kappa-1-k_1-k_2}\|\psi\|_{W^{k_1+k_2,\infty}(0,\infty)}^2.
\]
Because $c < R^2G^{-1}T^{-1+\epsilon_1}$, we have
\[
T^{2\kappa}c^{k+2\kappa}R^{-2k-4\kappa} < T^{2\kappa} G^{-(k+2\kappa)}T^{-k-2\kappa + \epsilon_1(k+2\kappa)}=T^{-k}T^{-(\theta-\epsilon_1)(k+2\kappa)} \leq 1,
\]
and so we obtain \eqref{int2}.
\end{proof}
\begin{lemma}\label{lem:phase}
Assume $r_1, r_2 \in ( R/l,lR )$, $|y-T|<T/2$, and $0<d<R^{\frac{1}{2}}$. Fix a constant $\epsilon_2>0$, and assume that $dR^{\epsilon_2}<R^2G^{-1}T^{-1+\epsilon_1}$, and that $dR^{\epsilon_2}<c<R^2G^{-1}T^{-1+\epsilon_1}$. Then for any nonnegative integers $k_1, k_2 \geq 0$, we have
\begin{equation}\label{int1}
\frac{\partial^{k_1+k_2}\varphi}{\partial {r_1}^{k_1}\partial {r_2}^{k_2}}\ll_{l, k_1,k_2} cT^2R^{-2-k_1-k_2}.
\end{equation}
Moreover, when $k_1+k_2 =1$ or $2$, we have
\begin{equation}\label{stat}
\varphi_{r_1} \sim_l \frac{cT^2}{R^3}, \varphi_{r_2} \sim_l \frac{cT^2}{R^3},
\end{equation}
\begin{equation}\label{end}
|\varphi_{r_ir_j}| \sim_l \frac{cT^2}{R^4} ~\text{ and }~ |\varphi_{r_1r_1}\varphi_{r_2r_2}-\varphi_{r_1r_2}^2| \gg_l \frac{c^2T^4}{R^8}.
\end{equation}
\end{lemma}
\begin{proof}
Consider the power series expansion
\begin{align*}
\Delta = r_1r_2\sqrt{1+\frac{d}{r_1}}\sqrt{1+\frac{d}{r_2}} &= r_1r_2\sum_{n_1,n_2=0}^\infty \binom{1/2}{n_1} \binom{1/2}{n_2} \left(\frac{d}{r_1}\right)^{n_1}\left(\frac{d}{r_2}\right)^{n_2}\\
&=\sum_{n_1,n_2=0}^\infty \binom{1/2}{n_1} \binom{1/2}{n_2} \frac{d^{n_1+n_2}}{r_1^{n_1-1}r_2^{n_2-1}},
\end{align*}
where $\binom{z}{n} := \frac{z(z-1) \ldots (z-n+1)}{n!}$. For $r_1, r_2 \in \lbrack R/l,lR\rbrack$, we therefore have
\begin{align*}
\Delta&=r_1r_2+O_l\left(dR\right)\\
\Delta_{r_i}&=r_{3-i}+\frac{d}{2}+O_l\left(d^2R^{-1}\right)\\
\Delta_{r_ir_i}&=O_l\left(d^2R^{-2}\right)\\
\Delta_{r_1r_2}&=1+O_l\left(d^2R^{-2}\right)\\
\frac{\partial^{k_1+k_2}\Delta}{\partial {r_1}^{k_1}\partial {r_2}^{k_2}}&=O_{l,k_1,k_2}\left(d^2R^{-k_1-k_2}\right) \tag{$k_1+k_2 \geq 3$}\\
\frac{\partial^{k_1+k_2}}{\partial {r_1}^{k_1}\partial {r_2}^{k_2}}(\Delta-r_1r_2-&\frac{dr_1}{2}-\frac{dr_2}{2}) =O_{l,k_1,k_2}\left(d^2R^{-k_1-k_2}\right) \tag{$k_1+k_2 \geq 0$}.
\end{align*}
For $x\gg GT^{1-\epsilon_1} \gg T^{1+\frac{\theta}{2}}$ and $T/2<y<3T/2$, we have
\[
\frac{\partial^m\alpha}{\partial x^m}(x) =O_m\left(T^2x^{-1-m}\right) \tag{$m \geq 0$}.
\]
Now note that $\frac{\partial^{k_1+k_2}}{\partial {r_1}^{k_1}\partial {r_2}^{k_2}} \alpha\left(\frac{4\pi \Delta}{c}\right)$ is a linear combination of terms of the form
\[
\alpha^{(m)}\left(\frac{4\pi \Delta}{c}\right) \prod_{j=1}^m \frac{\partial^{a_j+b_j}(\Delta/c)}{\partial r_1^{a_j} \partial r_2^{b_j}}
\]
with $1\leq m\leq k_1+k_2$, $\sum_{j=1}^m a_j=k_1$, and $\sum_{j=1}^m b_j =k_2$. Let $m_1$ and $m_2$ be the number of $j$ such that $a_j+b_j=1$ and $a_j+b_j=2$, respectively. Then
\begin{align*}
\alpha^{(m)}\left(\frac{4\pi \Delta}{c}\right) \prod_{j=1}^m \frac{\partial^{a_j+b_j}(\Delta/c)}{\partial r_1^{a_j} \partial r_2^{b_j}}
&\ll_{l,k_1,k_2} T^2 \left(\frac{\Delta}{c}\right)^{-1-m}c^{-m} R^{m_1} \frac{d^{2(m-m_1-m_2)}}{R^{k_1+k_2-m_1-2m_2}}\\
&\ll_l cT^2R^{-2-2m+m_1}  \frac{d^{2(m-m_1-m_2)}}{R^{k_1+k_2-m_1-2m_2}}\\
&=cT^2 R^{-2 - k_1-k_2}\left(\frac{d}{R}\right)^{2(m-m_1-m_2)}\\
&<cT^2 R^{-2 - k_1-k_2}.
\end{align*}
From this, we obtain \eqref{int1},
\begin{equation*}
\frac{\partial^{k_1+k_2}\varphi}{\partial {r_1}^{k_1}\partial {r_2}^{k_2}}= O_{l,k_1,k_2}( cT^2R^{-2-k_1-k_2})+O_{l,k_1,k_2}(c^{-1}d^2R^{-k_1-k_2})=O_{l,k_1,k_2}( cT^2R^{-2-k_1-k_2}),
\end{equation*}
where we used the assumption that $dR^{\epsilon_2}<c$.

Now to prove \eqref{stat} and \eqref{end}, we first compute
\begin{align*}
\alpha_x&=-c_1\frac{y^2}{x^2}+O\left(T^4x^{-4}\right)\\
\alpha_{xx}&=2c_1\frac{y^2}{x^3}+O\left(T^4x^{-5}\right).
\end{align*}
Then we have
\begin{align*}
\varphi_{r_i} &=\frac{4\pi\alpha_x\left(4\pi\Delta/c\right)}{c}\Delta_{r_i} +O_l(c^{-1}d^2 R^{-1})\\
&=\frac{- c_1 c y^2}{4\pi \Delta^2} \Delta_{r_i} + O\left(\frac{T^4 c^3 \Delta_{r_i}}{\Delta^4}\right)+O_l(c^{-1}d^2 R^{-1})\\
&=\frac{  c y^2}{2\pi \Delta^2} \Delta_{r_i} + O_l\left(T^4 c^3 R^{-7}\right) + O_l(c^{-1}d^2R^{-1}).
\end{align*}
The leading term is $\sim cT^2R^{-3}$, and the error term is $o(cT^2R^{-3})$ since we assumed that $ dR^{\epsilon_2}<c<R^2G^{-1}T^{-1+\epsilon_1}$. From this, we conclude \eqref{stat}. Likewise, we obtain \eqref{end} from the following computation
\begin{align*}
\varphi_{r_ir_j}&=\frac{16 \pi^2 \alpha_{xx}\left(4\pi \Delta/c\right)}{c^2}\Delta_{r_i}\Delta_{r_j}+\frac{4 \pi \alpha_{x}\left(4\pi \Delta/c\right)}{c}\Delta_{r_ir_j}+O_l\left(\frac{d^2}{cR^2}\right)\\
&=\frac{8 \pi^2 \alpha_{xx}\left(4\pi \Delta/c\right)}{c^2}\left(2\Delta_{r_i}\Delta_{r_j}-\Delta\Delta_{r_ir_j}\right)+O_l\left(\frac{\Delta_{r_ir_j}T^4c^3}{R^8}+\frac{c}{R^{2+2\epsilon_2}}\right)\\
&=\frac{8 \pi^2 \alpha_{xx}\left(4\pi \Delta/c\right)}{c^2}\left(\left(2-\delta_{i,j}\right)r_{3-i}r_{3-j}+O_l\left(dR\right)\right)+O_l\left(\frac{T^4c^3}{R^8}+\frac{c}{R^{2+2\epsilon_2}}\right).\qedhere
\end{align*}
\end{proof}

\begin{lemma}\label{lem:po}
Assume that $|y-T|<T/2$, and $0<d<R^{\frac{1}{2}}$. Fix a constant $\epsilon_2>0$, and assume that $dR^{\epsilon_2}<R^2G^{-1}T^{-1+\epsilon_1}$, and that $dR^{\epsilon_2}<c<R^2G^{-1}T^{-1+\epsilon_1}$. Then for any nonnegative integers $k_1, k_2 \geq 0$, we have
\begin{equation*}
\frac{\partial^{k_1+k_2}f_c}{\partial {r_1}^{k_1}\partial {r_2}^{k_2}}\ll_{l,k_1,k_2} \frac{c^{1/2}}{R}T^{-(k_1+k_2)\left(3\theta-1\right)/4}\|\psi\|_{W^{k_1+k_2,\infty}}^2.
\end{equation*}
\end{lemma}
\begin{proof}
Recall that $f_c= g_c \exp (i\varphi)$. Hence
\[
\frac{\partial^{k_1+k_2}f_c}{\partial {r_1}^{k_1}\partial {r_2}^{k_2}}
\]
is a linear combination of
\[
\frac{\partial^{n_1+n_2}g_c}{\partial {r_1}^{n_1}\partial {r_2}^{n_2}}(r_1,r_2) e^{i\varphi(r_1,r_2)} \prod_{j=1}^m \frac{\partial^{a_j+b_j}\varphi}{\partial r_1^{a_j}\partial r_2^{b_j}} (r_1,r_2).
\]
where $n_1+\sum_{j=1}^m a_j = k_1$, $n_2+\sum_{j=1}^m b_j = k_2$, and $ 0\leq m \leq k_1+k_2-n_1-n_2$. From Lemma \ref{lem:non} and \eqref{int1} of Lemma \ref{lem:phase},
\begin{align*}
\frac{\partial^{n_1+n_2}g_c}{\partial {r_1}^{n_1}\partial {r_2}^{n_2}}(r_1,r_2) e^{i\varphi(r_1,r_2)} \prod_{j=1}^m \frac{\partial^{a_j+b_j}\varphi}{\partial r_1^{a_j}\partial r_2^{b_j}} (r_1,r_2) &\ll_{l,k_1,k_2} c^{1/2}R^{-1-n_1-n_2}\|\psi\|_{W^{k_1+k_2,\infty}(0,\infty)}^2 \prod_{j=1}^m cT^2 R^{-2-a_j-b_j}\\
&= \frac{c^{1/2}}{R} \left(\frac{cT^2}{R^2}\right)^m R^{-k_1-k_2}.
\end{align*}
Because $cT^2/R^2 \gg c > dR^{\epsilon_2}$, for all sufficiently large $R$, the last expression is bounded from above by
\[
\frac{c^{1/2}}{R} \left(\frac{cT^2}{R^2}\right)^{k_1+k_2}R^{-k_1-k_2} = \frac{c^{1/2}}{R} \left(\frac{cT^2}{R^3}\right)^{k_1+k_2}.
\]
Now note that for $R$ that satisfies $R^2>GT^{1-\epsilon_1}$, we have
\[
\frac{cT^2}{R^3} \ll \frac{T^{1+\epsilon_1}}{GR} \ll\left(\frac{T^{1+3\epsilon_1}}{G^3}\right)^{1/2} \ll T^{\left(1+3\epsilon_1-3\theta\right)/2} \ll T^{-\left(3\theta-1\right)/4}.
\]
from the assumption that $\frac{3\theta-1}{6}>\epsilon_1$.
\end{proof}

\subsubsection{The Poisson summation formula and completion of the proof}
Applying the Poisson summation formula for the sum in $r_1$ and $r_2$, we get
\[
\sum_{r_1,r_2}f_c\left(r_1,r_2\right)e_c\left(-ur_1-vr_2\right)=\sum_{j,k}B\left(j,k\right)
\]
where
\begin{align*}
B\left(j,k\right)&=\iint f_c\left(r_1,r_2\right)e_c\left(-ur_1-vr_2\right)e\left(jr_1+kr_2\right)dr_1dr_2\\
&=\iint f_c\left(r_1,r_2\right)e\left(\left(j-\frac{u}{c}\right)r_1+\left(k-\frac{v}{c}\right)r_2\right)dr_1dr_2.
\end{align*}

We first show that the contribution coming from $(j,k) \neq (0,0)$ is negligible.
\begin{lemma}\label{fin:3}
Assume that $|y-T|<T/2$, and $0<d<R^{\frac{1}{2}}$. Fix a constant $\epsilon_2>0$, and assume that $dR^{\epsilon_2}<R^2G^{-1}T^{-1+\epsilon_1}$, and that $dR^{\epsilon_2}<c<R^2G^{-1}T^{-1+\epsilon_1}$. Then we have
\[
\sum_{\substack{j,k\\ (j,k)\neq (0,0)}}B(j,k) = O_{l,\theta}\left( T^{-20}\|\psi\|_{W^{m_0,\infty}}^2\right),
\]
where $m_0=\lceil\frac{100}{3\theta-1}\rceil$.
\end{lemma}
\begin{proof}
We first give an upper bound for $B(j,k)$ under the assumption that $\max \{j,k\}=j>0$. By integration by parts, we have
\[
|B(j,k)| = \left|\left(2\pi\left(j-\frac{u}{c}\right)\right)^{-m}\iint \left(\frac{\partial^m}{\partial r_1^m}f_c\left(r_1,r_2\right)\right)e\left(\left(j-\frac{u}{c}\right)r_1+\left(k-\frac{v}{c}\right)r_2\right)dr_1dr_2\right|.
\]
Because $f_c$ is supported in $\mathcal{B}=(R/l,lR)\times(R/l,lR)$, we may assume that the integral is taken over $\mathcal{B}$. From Lemma \ref{lem:po},
\[
\iint_\mathcal{B} \left(\frac{\partial^m}{\partial r_1^m}f_c\left(r_1,r_2\right)\right)e\left(\left(j-\frac{u}{c}\right)r_1+\left(k-\frac{v}{c}\right)r_2\right)dr_1dr_2\ll_{l,m}  c^{1/2}RT^{-m\left(3\theta-1\right)/4}\|\psi\|_{W^{m,\infty}}^2.
\]
From the assumption that $|u|,|v| \leq c/2$, we now have
\[
B(j,k) \ll_{l,m} (2|j|-1)^{-m} c^{1/2}RT^{-m\left(3\theta-1\right)/4}\|\psi\|_{W^{m,\infty}}^2,
\]
and hence
\[
B(j,k) \ll_{l,m} (2\max\{|j|,|k|\}-1)^{-m} c^{1/2}RT^{-m\left(3\theta-1\right)/4}\|\psi\|_{W^{m,\infty}}^2.
\]
Now by taking $m=m_0=\lceil\frac{100}{3\theta-1}\rceil>50$, we conclude that
\[
\sum_{\substack{j,k\\ (j,k)\neq (0,0)}} B(j,k) \ll_{l,\theta} c^{1/2} R T^{-25} \|\psi\|_{W^{m_0,\infty}}^2 \ll_{l,\theta} T^{-20}\|\psi\|_{W^{m_0,\infty}}^2.
\]
\end{proof}

For $B\left(0,0\right)$, observe from \eqref{stat} of Lemma \ref{lem:phase} that
\[
B(0,0)=\iint g_c\left(r_1,r_2\right)e^{i\varphi\left(r_1,r_2\right)-\frac{2\pi i}{c}\left(ur_1+vr_2\right)}dr_1dr_2
\]
has stationary phase in both $r_1$ and $r_2$ variables only when
\begin{equation}\label{eq2}
u\sim_l \frac{c^2T^2}{R^3}~\text{and}~ v \sim_l \frac{c^2T^2}{R^3}
\end{equation}
are satisfied. Otherwise, we perform integration by parts to show that $B(0,0)$ is negligibly small.
\begin{lemma}\label{fin:2}
Assume that $|y-T|<T/2$, and $0<d<R^{\frac{1}{2}}$. Fix a constant $\epsilon_2>0$, and assume that $dR^{\epsilon_2}<R^2G^{-1}T^{-1+\epsilon_1}$, and that $dR^{\epsilon_2}<c<R^2G^{-1}T^{-1+\epsilon_1}$. Then there exists a constant $\eta$ depending only on $l$ such that if either $u$ or $v$ is not in the range
\[
\lbrack \eta^{-1}\frac{c^2T^2}{2\pi R^3},\eta\frac{c^2T^2}{2\pi R^3}\rbrack
\]
then
\[
B(0,0)= O_{l, \epsilon_2}\left( T^{-40}\|\psi\|_{W^{n_0,\infty}\left(0,\infty\right)}^2\right),
\]
where $n_0 = \lceil \frac{100}{\epsilon_2} \rceil$.
\end{lemma}
\begin{proof}
From \eqref{stat} of Lemma \ref{lem:phase}, there exists a constant $\eta_0>0$ depending only on $l$ such that
\[
\eta_0^{-1} \frac{cT^2}{R^3}< \varphi_{r_i} < \eta_0 \frac{cT^2}{R^3}
\]
for $r_1,r_2 \in (R/l,lR)$. Let $\eta = 2\eta_0$ and assume without loss of generality that $u$ is not in the range
\[
\lbrack \eta^{-1}\frac{c^2T^2}{2\pi R^3},\eta\frac{c^2T^2}{2\pi R^3}\rbrack.
\]
Let $I_n(r_1,r_2)$ for $n \geq 0$ be given by $I_0=g_c$, and
\[
I_n(r_1,r_2) = i \frac{\partial}{\partial r_1}\frac{I_{n-1}(r_1,r_2)}{\varphi(r_1,r_2)_{r_1} - \frac{2\pi u}{c}}
\]
for $n \geq 1$, so that
\[
B(0,0) = \iint I_n(r_1,r_2)e^{i\varphi\left(r_1,r_2\right)-\frac{2\pi i}{c}\left(ur_1+vr_2\right)}dr_1dr_2.
\]
Then $I_n$ is a linear combination of
\[
\frac{\partial^{a_0}}{\partial r_1^{a_0}}g_c(r_1,r_2)\prod_{j=1}^n \frac{\partial^{a_j}}{\partial r_1^{a_j}}\frac{1}{\varphi(r_1,r_2)_{r_1} - \frac{2\pi u}{c}}
\]
where $\sum_{j=0}^n a_j=n$. For each $j$,
\[
\frac{\partial^{a_j}}{\partial r_1^{a_j}}\frac{1}{\varphi(r_1,r_2)_{r_1} - \frac{2\pi u}{c}}
\]
is a linear combination of
\[
\frac{1}{(\varphi(r_1,r_2)_{r_1} - \frac{2\pi u}{c})^{b_j+1}}\prod_{k=1}^{b_j}\frac{\partial^{b_{jk}}}{\partial r_1^{b_{jk}}}\left(\varphi(r_1,r_2)_{r_1} - \frac{2\pi u}{c}\right)
\]
where $\sum_{k=1}^{b_j}b_{jk}=a_j$, $b_j \leq a_j$ and $b_{jk} \geq 1$. Observe that
\[
\frac{1}{\varphi(r_1,r_2)_{r_1} - \frac{2\pi u}{c}} \ll_l \frac{R^3}{cT^2}
\]
from the assumption on $u$, and that
\[
\frac{\partial^{b_{jk}}}{\partial r_1^{b_{jk}}}\left(\varphi(r_1,r_2)_{r_1} - \frac{2\pi u}{c}\right)=\frac{\partial^{b_{jk}}}{\partial r_1^{b_{jk}}}\varphi(r_1,r_2)_{r_1}.
\]
Hence we have from Lemma \ref{lem:phase} that
\begin{align*}
&\frac{1}{(\varphi(r_1,r_2)_{r_1} - \frac{2\pi u}{c})^{b_j+1}}\prod_{k=1}^{b_j}\frac{\partial^{b_{jk}}}{\partial r_1^{b_{jk}}}\left(\varphi(r_1,r_2)_{r_1} - \frac{2\pi u}{c}\right)\\
\ll_{l,a_j} & \left(\frac{R^3}{cT^2}\right)^{b_j+1} \prod_{k=1}^{b_j} cT^2R^{-3-b_{jk}}\\
=& \left(\frac{R^3}{cT^2}\right)^{b_j+1} c^{b_j}T^{2b_j}R^{-3b_j-a_j}\\
=& c^{-1}R^{3-a_j}T^{-2}.
\end{align*}
Now we apply Lemma \ref{lem:non} so that
\begin{align*}
\frac{\partial^{a_0}}{\partial r_1^{a_0}}g_c(r_1,r_2)\prod_{j=1}^n \frac{\partial^{a_j}}{\partial r_1^{a_j}}\frac{1}{\varphi(r_1,r_2)_{r_1} - \frac{2\pi u}{c}} &\ll_{l,n} c^{1/2} R^{-1-a_0}\|\psi\|_{W^{n,\infty}\left(0,\infty\right)}^2\prod_{j=1}^n c^{-1}R^{3-a_j}T^{-2}\\
&=c^{1/2}R^{-1} (c^{-1}R^2T^{-2})^n \|\psi\|_{W^{n,\infty}\left(0,\infty\right)}^2\\
&\ll c^{1/2}R^{-1} c^{-n}\|\psi\|_{W^{n,\infty}\left(0,\infty\right)}^2,
\end{align*}
and therefore we have
\[
B(0,0) \ll_{l,n} c^{1/2}R c^{-n}\|\psi\|_{W^{n,\infty}\left(0,\infty\right)}^2.
\]
Because we assumed that $c> dR^{\epsilon_2} > R^{\epsilon_2}$, by taking $n = n_0 = \lceil \frac{100}{\epsilon_2} \rceil $, we conclude that
\[
B(0,0) \ll_{l, \epsilon_2} T^{-40}\|\psi\|_{W^{n_0,\infty}\left(0,\infty\right)}^2
\]
where we used $R^2 > GT^{1-\epsilon_1} \gg T$.
\end{proof}

In order to treat the remaining case for which \eqref{eq2} holds, we apply integration by parts to get
\begin{align*}
B(0,0)=&\iint f_c\left(r_1,r_2\right)e_c\left(-ur_1-vr_2\right)dr_1dr_2 \\
= &\iint \int_{0}^{r_1}\int_{0}^{r_2} e^{i\varphi\left(t_1,t_2\right)}e_c\left(-ut_1-vt_2\right)dt_1dt_2 g_c(r_1,r_2)_{r_1r_2} dr_1dr_2\\
\ll & \sup_{R/l<r_1<lR, R/l<r_2<lR}\left| \int_{0}^{r_1}\int_{0}^{r_2} e^{i\varphi\left(t_1,t_2\right)-\frac{2\pi i}{c}\left(ut_1+vt_2\right)}dt_1dt_2\right|\iint|g_c\left(r_1,r_2\right)_{r_1r_2}|dr_1dr_2.
\end{align*}
Note that the phase function $f(t_1,t_2)=\varphi\left(t_1,t_2\right)-\frac{2\pi }{c}\left(ut_1+vt_2\right)$ satisfies the estimates in the following Lemma with $\lambda = cT^2R^{-4}$, by Lemma \ref{lem:phase}.
\begin{lemma}[\cite{tit1934}]
Let $f\left(t_1,t_2\right)$ be a real and algebraic function defined in a rectangle $D=[a,b]\times [c,d] \subset \mathbb{R}^2$. Assume throughout $D$ that
\[
|f_{t_it_i}|\sim\lambda \text{ for } i=1,2,~|f_{t_1t_2}| \ll \lambda, \text{ and } \left|\frac{\partial \left(f_{t_1},f_{t_2}\right)}{\partial \left(t_1,t_2\right)}\right| \gg \lambda^2;
\]
then
\[
\iint_D e^{if\left(t_1,t_2\right)}dt_1dt_2 \ll \frac{1+|\log\left(b-a\right)|+|\log\left(d-c\right)|+|\log \lambda|}{\lambda}.
\]
\end{lemma}
We therefore have
\[
\int_{0}^{r_1}\int_{0}^{r_2} e^{i\varphi\left(t_1,t_2\right)-\frac{2\pi i}{c}\left(ut_1+vt_2\right)}dt_1dt_2\ll_l \log T c^{-1}T^{-2}R^4
\]
uniformly in $R/l<r_1<lR$ and $R/l<r_2<lR$. We also know that
\[
\iint|g_c\left(r_1,r_2\right)_{r_1r_2}|dr_1dr_2 \ll_l c^{1/2}R^{-1}\|\psi\|_{W^{2,\infty}(0,\infty)}^2
\]
from Lemma \ref{lem:non}. Combining these estimates, we obtain:
\begin{lemma} \label{fin:1}
Assume that $|y-T|<T/2$, and $0<d<R^{\frac{1}{2}}$. Fix a constant $\epsilon_2>0$, and assume that $dR^{\epsilon_2}<R^2G^{-1}T^{-1+\epsilon_1}$, and that $dR^{\epsilon_2}<c<R^2G^{-1}T^{-1+\epsilon_1}$. Then we have
\[
B(0,0) = O_{l} \left(c^{-1/2}R^3 T^{-2}\log T\|\psi\|_{W^{2,\infty}(0,\infty)}^2\right).
\]
\end{lemma}
Lemma \ref{lem:poisson} now follows immediately from Lemma \ref{fin:3}, Lemma \ref{fin:2}, and Lemma \ref{fin:1}.

\section{Quantitative quantum ergodicity-II}\label{qqe2}
\subsection{The case \texorpdfstring{$m=0$}{m=0}}
Let $G\left(s\right)$ be the Mellin transform of $\psi \in C_0^\infty(0,\infty)$ whose support is in $(1/l,l)$:
\[
G\left(s\right)=\int_0^\infty \psi\left(y\right)y^{s-1}dy.
\]
Then $G(s)$ is entire function in $s$, and the Mellin inversion formula is given by
\[
\psi(X)=\frac{1}{2\pi i } \int_{\sigma} X^{-s} G(s) ds,
\]
where $(\sigma)$ is the contour $(\sigma-i\infty, \sigma+i\infty)$. Therefore we have that
\[
\sum_{n =1}^\infty \rho_\phi\left(n\right)^2 \psi\left(\frac{\pi |n|}{X}\right) = \frac{1}{2 \pi i} \int_{\left(2\right)} \sum_{n\geq 1} \frac{\rho_\phi\left(n\right)^2}{n^s}\left(\frac{X}{\pi}\right)^s G\left(s\right) ds=\frac{1}{2 \pi i} \int_{\left(2\right)} L(s,\phi \times \phi)\left(\frac{X}{\pi}\right)^s G\left(s\right) ds,
\]
and by shifting the contour to $(1/2)$, we get
\[
\frac{1}{2 \pi i} \int_{\left(1/2\right)} L(s,\phi \times \phi) \left(\frac{X}{\pi}\right)^s G\left(s\right) ds+\frac{12}{\pi^3}X\int_0^\infty \psi(y)dy,
\]
since  $L(s,\phi \times \phi) =\sum_{n\geq 1} \frac{\rho_\phi\left(n\right)^2}{n^s}$ has a simple pole at $s=1$ whose residue is $12\pi^{-2}$ \cite{MR1431508}.

Note that we have a factorization
\begin{align*}
\sum_{n\geq 1} \frac{\rho_\phi\left(n\right)^2}{n^s} = \rho_\phi(1)^2 \sum_{n\geq 1} \frac{\lambda_\phi\left(n\right)^2}{n^s} = \rho_\phi(1)^2 \frac{\zeta(s)}{\zeta(2s)} L\left(s, {\rm sym}^2 \phi\right)
\end{align*}
where $\zeta(s)$ is the Riemann zeta function and
\[
L\left(s, {\rm sym}^2 \phi\right)= \sum_{n=1}^\infty \frac{c_\phi(n)}{n^s}= \sum_{n=1}^\infty \frac{\sum_{l^2k=n}\lambda_\phi(k^2)}{n^s}
\]
is the symmetric square $L$-function attached to $\phi$, so we have
\[
\sum_{n =1}^\infty \rho_\phi\left(n\right)^2 \psi\left(\frac{\pi |n|}{X}\right)-\frac{12}{\pi^3}X\int_0^\infty \psi(y)dy =\frac{\rho_\phi(1)^2}{2 \pi i} \int_{\left(1/2\right)}  \frac{\zeta(s)}{\zeta(2s)} L\left(s, {\rm sym}^2 \phi\right)\left(\frac{X}{\pi}\right)^s G\left(s\right) ds.
\]

Then it is known that $L(s,{\rm sym} \phi^2)$ is entire and that the following functional equation is satisfied \cite{shim}:
\[
\Lambda(s,{\rm sym} \phi^2)=\Lambda(1-s,{\rm sym} \phi^2),
\]
where $\Lambda(s,{\rm sym} \phi^2)$ is given by
\begin{equation}\label{finfinfin}
\Lambda(s,{\rm sym} \phi^2) = \pi^{-\frac{3}{2}s}\Gamma\left(\frac{s}{2}\right)\Gamma\left(\frac{s}{2}+it_\phi\right)\Gamma\left(\frac{s}{2}-it_\phi\right)L(s,{\rm sym} \phi^2)=\gamma(s,{\rm}\phi^2)L(s,{\rm sym} \phi^2).
\end{equation}
We use approximate functional equation (equation below (32) \cite{sarluo2} or Theorem 5.3 \cite{ant}) to represent $L(s,{\rm sym} \phi^2)$ as a smooth finite sum of $c_\phi(n)n^{-s}$ of length at most $t_\phi^{1+\epsilon}$.
\[
L(s,{\rm sym} \phi^2)\approx \sum_{n=1}^\infty \frac{c_\phi(n)}{n^s} V_s^1(\frac{n}{t_\phi})+\sum_{n=1}^\infty \frac{c_\phi(n)}{n^{1-s}} V_s^2(\frac{n}{t_\phi}).
\]
Plugging this into \eqref{finfinfin}, we get a smooth sum of $\lambda_\phi(n^2)$ whose length is at most $X^{1+\epsilon}$. Now the proof of the case when $m=0$ follows by following the proof of the case when $m\neq 0$.

\section{Proof of the corollaries}
\subsection{Quantitative Quantum Ergodicity}
In this section, we prove Corollary \ref{cor122}. For this purpose, we first approximate
\[
\int_\mathbb{X} P_{m,h}\left(z\right) |\phi\left(z\right)|^2 dV
\]
by a shifted convolution sum.
\begin{theorem}\label{app}
Fix $ 1/2 >\kappa >0$ and $L>1$. For any given $h \in C_0^\infty(1/L,L)$, we have the following estimate uniformly in $0 \leq m \ll t_\phi^\kappa$
\begin{multline*}
\int_\mathbb{X} P_{m,h}\left(z\right) |\phi\left(z\right)|^2 dV = \frac{\pi}{ t_\phi}\sum_{n \neq 0,-h}\rho_\phi(n){\rho_\phi(n+m)}k_{m,h}\left(\frac{|n|}{t_\phi}\right)\\
+O_{A,L,\epsilon}\left(\|h\|_{W^{A+1,\infty}}\left(t_\phi^{-1/2+3\kappa+\epsilon}+t^{-\kappa A+\epsilon}\right)\right),
\end{multline*}
for any $A>0$ and $\epsilon>0$.

Here $k_{m,h}$ is given by:
\begin{align*}
k_{m,h}(1/u) = u\int_0^{u} h\left(\frac{y}{2\pi}\right) \frac{\cos \left(m\sqrt{u^2-y^2}\right)}{\sqrt{u^2-y^2}}\frac{dy}{y}.
\end{align*}
\end{theorem}
To prove Theorem \ref{app}, we need some lemmas.
\begin{lemma}\label{lem:fin}
\begin{equation}\label{poin}
\int_\mathbb{X} P_{m,h}\left(z\right) |\phi\left(z\right)|^2 dV = \sum_{n \neq 0,-m}\rho_\phi(n){\rho_\phi(n+m)}G_{t_\phi,h}(n,n+m)
\end{equation}
where $G_{t,h}$ is given by
\begin{equation}\label{int}
G_{t,h}(n_1,n_2)= \cosh (\pi t) \int_0^\infty K_{it}(2\pi|n_1|y)K_{it}(2\pi |n_2|y) h(y) y^{-1}dy.
\end{equation}
\end{lemma}
\begin{proof}
Recall that the Poincar\'e series is given by
\[
P_{m,h}(z)=\sum_{\gamma \in \Gamma_\infty \backslash \Gamma} h(\mathrm{Im}(\gamma z))e^{2\pi i m \mathrm{Re}(\gamma z)},
\]
for $m\in \mathbb{Z}$ and $h \in C_0^\infty \left(0,\infty\right)$. By unfolding the integral, we may rewrite
\[
\int_\mathbb{X} P_{m,h}\left(z\right) |\phi\left(z\right)|^2 dV
\]
as follows:
\begin{align*}
\int_\mathbb{X} P_{m,h}\left(z\right) |\phi\left(z\right)|^2 dV &= \int_\mathbb{X} \sum_{\gamma \in \Gamma_\infty \backslash \Gamma} h(\mathrm{Im}(\gamma z))e^{2\pi i m \mathrm{Re}(\gamma z)} |\phi\left(z\right)|^2 dV\\
&=\int_0^\infty \int_{-1/2}^{1/2}h(y)e(mx) |\phi(x+iy)|^2 \frac{dxdy}{y^2}\\
&= \cosh(\pi t_\phi)\int_0^\infty h(y) \int_{-1/2}^{1/2}e(mx) \left|\sum_{n \neq 0} \rho_\phi(n) K_{it_\phi}(2\pi |n|y) e(nx)\right|^2 \frac{dxdy}{y}\\
&= \cosh(\pi t_\phi)\int_0^\infty h(y) \sum_{n \neq 0,-m} \rho_\phi(n)\rho_\phi(n+m) K_{it_\phi}(2\pi |n|y)K_{it_\phi}(2\pi |n+m|y)  \frac{dy}{y}\\
&=  \sum_{n \neq 0,-m} \rho_\phi(n)\rho_\phi(n+m) \cosh(\pi t_\phi)\int_0^\infty h(y)K_{it_\phi}(2\pi |n|y)K_{it_\phi}(2\pi |n+m|y)  \frac{dy}{y}.\qedhere
\end{align*}
\end{proof}

Now assume for the rest of the section that $h$ is compactly supported smooth function supported in $(1/L,L)$ for some $L>1$, and $m \ll t_\phi^\kappa$ with a fixed constant $0 < \kappa < 1$.
\begin{lemma}\label{lem:fin1}
For any fixed $0<\delta<1-\kappa$, we have
\begin{align*}
&\sum_{n \neq 0,-m}\rho_\phi(n){\rho_\phi(n+m)}G_{t_\phi,h}(n,n+m)\\
=&\sum_{|n| \gg t_\phi^{1-\delta}}\rho_\phi(n){\rho_\phi(n+m)}G_{t_\phi,h}(n,n+m)+O_{\epsilon,L}(t_\phi^{-\delta+\epsilon}\|h\|_{L^\infty}).
\end{align*}
\end{lemma}
\begin{proof}
We need the following estimate for Fourier coefficients \cite{iwan90}:
\begin{equation}\label{partial}
\sum_{ n < X} |\rho_\phi(n)|^2 \ll_\epsilon Xt_\phi^\epsilon.
\end{equation}
Observe from the asymptotic expansion of the $K$-Bessel function \cite{er81} (or Corollary 3.2 of \cite{gzs}) that when $n > Lt$, $G_{t,h}(n,n)$ is negligible, whereas for $n\leq Lt$,
\begin{align*}
G_{t,h}(n,n) &= \cosh( \pi t) \int_0^\infty h(y) K_{it}^2(2\pi n y) \frac{dy}{y}\\
&= \cosh( \pi t) \int_0^\infty h(y/(2\pi n)) K_{it}^2(y) \frac{dy}{y}\\
&\ll \int_0^{t-100t^{1/3}}|h(y/(2\pi n))| (t^2-y^2)^{-1/2}\frac{dy}{y}+ O(t^{-4/3} \|h\|_{L^\infty}),
\end{align*}
and
\begin{align*}
\int_0^{t-100t^{1/3}}|h(y/(2\pi n))| (t^2-y^2)^{-1/2}\frac{dy}{y}&\leq \int_0^{t}|h(y/(2\pi n))| (t-y)^{-1/2}(t+y)^{-1/2}\frac{dy}{y}\\
&\leq \frac{1}{\sqrt{t}}\int_0^{t}|h(y/(2\pi n))| (t-y)^{-1/2}\frac{dy}{y}\\
&= \frac{1}{\sqrt{2\pi n t}}\int_0^{t/(2\pi n)}|h(y)| (t/(2\pi n)-y)^{-1/2}\frac{dy}{y}.
\end{align*}
We first consider the case when $t/(2\pi n) >2L$, for which we have
\begin{align*}
\frac{1}{\sqrt{2\pi n t}}\int_0^{t/(2\pi n)}|h(y)| (t/(2\pi n)-y)^{-1/2}\frac{dy}{y}&=\frac{1}{\sqrt{2\pi n t}}\int_{1/L}^{L}|h(y)| (t/(4\pi n)+(t/(4\pi n)-y))^{-1/2}\frac{dy}{y}\\
 &\ll_L \frac{\|h(y)\|_{L^\infty}}{t}.
\end{align*}
When $t/ (4\pi L) n<Lt$, we have
\begin{align*}
\frac{1}{\sqrt{2\pi n t}}\int_0^{t/(2\pi n)}|h(y)| (t/(2\pi n)-y)^{-1/2}\frac{dy}{y}&\leq \frac{L\|h(y)\|_{L^\infty}}{\sqrt{2\pi n t}}\int_0^{t/(2\pi n)} (t/(2\pi n)-y)^{-1/2}dy\\
&\ll_L \frac{\|h(y)\|_{L^\infty}}{t}.
\end{align*}
Therefore for any $ t_\phi^{\kappa+\epsilon} \ll X \ll t_\phi$,
\begin{equation*}
\sum_{0<|n|,|n+m| <X} \rho_\phi(n){\rho_\phi(n+m)}G_{t_\phi,h}(n,n+m) \ll_{L,\epsilon} Xt_\phi^{-1+\epsilon}\|h\|_{L^\infty},
\end{equation*}
so, for any fixed $0<\delta<1-\kappa$, we have
\begin{align*}
&\sum_{n \neq 0,-m}\rho_\phi(n){\rho_\phi(n+m)}G_{t_\phi,h}(n,n+m)\\
=&\sum_{|n| \gg t_\phi^{1-\delta}}\rho_\phi(n){\rho_\phi(n+m)}G_{t_\phi,h}(n,n+m)+O_{\epsilon,L}(t_\phi^{-\delta+\epsilon}\|h\|_{L^\infty}).\qedhere
\end{align*}
\end{proof}

\begin{lemma}\label{lem:fin2}
For $0 < \kappa+\delta < 1$, the following holds uniformly in $t^{1-\delta}\ll n \ll Lt$ and in $0 \leq m \ll t^\kappa$:
\[
G_{t,h}(n,n+m)=\frac{\pi}{t}k_m\left(\frac{\sqrt{n(n+m)}}{t}\right)+ O_{A,L}\left(\frac{\|h\|_{W^{A+1,\infty}}}{\sqrt{n(n+m)}}\left(t^{-2+3\delta+3\kappa}+t^{-\kappa A}\right)\right).
\]
Here $k_m(u)$ is given by:
\[
k_m(1/u)=u\int_0^u \frac{\cos\left(m\sqrt{u^2-y^2}\right)}{\sqrt{u^2-y^2}} h\left(\frac{y}{2\pi}\right)y^{-1}dy.
\]
\end{lemma}
\begin{proof}
Let
\[
f(r) = \int_0^\infty h(y) \cos (2\pi ry)  \frac{dy}{y}.
\]
Then by the formula from \cite{ryzhik}
\[
\cosh(\pi t)\int_0^\infty K_{it}(2\pi n_1 y)K_{it}(2\pi n_2 y) \cos(2\pi r y) dy
= \frac{\pi}{8\sqrt{n_1 n_2}} P_{-\frac{1}{2}+it}\left(\frac{n_1^2+n_2^2+r^2}{2n_1 n_2}\right)
\]
and by the asymptotic expansion of the Legendre function $P_{-1/2+it}$ from \cite{du91}, we have
\begin{align*}
&\cosh (\pi t)\int_0^\infty K_{it}(2\pi ny)K_{it}(2\pi (n+m)y) h(y)\frac{dy}{y}\\
&=\frac{\pi}{2\sqrt{n(n+m)}}\int_0^\infty  f(r)P_{-\frac{1}{2}+it}\left(1+\frac{m^2+r^2}{2n(n+m)}\right)dr\\
&=\frac{\pi}{2\sqrt{n(n+m)}}\int_0^\infty  f(r)\left(\frac{\ln (s+\sqrt{s^2-1})}{\sqrt{s^2-1}}\right)^{\frac{1}{2}}J_0\left(t\ln \left(s+\sqrt{s^2-1}\right)\right)dr\\
&+O\left(\frac{\|f\|_{L^1}}{t\sqrt{n(n+m)}}\right).
\end{align*}
Here $s=1+\frac{m^2+r^2}{2n(n+m)}=1+\alpha$. Since $f(r)=O_{A,L}((1+|r|)^{-A}\|h\|_{W^{A,\infty}})$, we may assume that the integral is taken in the range $r \ll t^\kappa$ plus an error term of $O_{A,L}(t^{-\kappa A}\|h\|_{W^{A+1,\infty}}))$. For such $r$, $\alpha =O_L(t^{2\kappa+2\delta-2})=o_{\kappa,\delta,L}(1)$. Then we have
\begin{align*}
\ln (s+\sqrt{s^2-1})&=\ln \left(1+\alpha +\sqrt{\alpha^2+2\alpha}\right)\\
&= \alpha +\sqrt{\alpha^2 + 2\alpha} - \frac{1}{2}\left(\alpha +\sqrt{\alpha^2 + 2\alpha}\right)^2 + O(\alpha^{3/2})\\
&=\sqrt{\alpha^2 + 2\alpha} - \alpha^2 - \alpha\sqrt{\alpha^2+2\alpha} +O(\alpha^{3/2})\\
&=\sqrt{\alpha^2 + 2\alpha} +O(\alpha^{3/2}),
\end{align*}
and
\begin{align*}
\sqrt{\alpha^2 + 2\alpha} &= \sqrt{2\alpha}\sqrt{1+\frac{\alpha}{2}}\\
&= \sqrt{2\alpha} + O\left(\alpha^{3/2}\right)\\
&= \sqrt{\frac{m^2+r^2}{n(n+m)}}+O\left(\alpha^{3/2}\right).
\end{align*}
From $\alpha = O\left(\frac{t^{2\kappa}}{n^2}\right)$, we conclude that
\begin{align*}
\frac{\ln (s+\sqrt{s^2-1})}{\sqrt{s^2-1}}&=1+O\left(\frac{t^{2\kappa}}{n^2}\right)\\
t\ln (s+\sqrt{s^2-1})&= \frac{t}{\sqrt{n(n+m)}}\sqrt{m^2+r^2}+O\left(t^{-2+3\delta+3\kappa}\right),
\end{align*}
hence
\begin{align*}
&\frac{\pi}{2\sqrt{n(n+m)}}\int_0^\infty  f(r)\left(\frac{\ln (s+\sqrt{s^2-1})}{\sqrt{s^2-1}}\right)^{\frac{1}{2}}J_0(t\ln (s+\sqrt{s^2-1}))dr \\
=&\frac{\pi}{2\sqrt{n(n+m)}}\int_0^\infty  f(r)J_0\left(\frac{t}{\sqrt{n(n+m)}}\sqrt{m^2+r^2}\right)dr+O_{A,L}\left(\frac{t^{-2+3\delta+3\kappa}\|h\|_{L^\infty}}{n}\right) .
\end{align*}
Now we use the following formula from \cite{ryzhik}
\[
2\int_0^{\frac{v}{2\pi}}\frac{\cos\left(h\sqrt{v^2-4\pi^2y^2}\right)}{\sqrt{v^2-4\pi^2y^2}} \cos(2\pi ry) dy=J_0(v\sqrt{h^2+r^2}),
\]
to obtain
\begin{align*}
&u\int_0^\infty f(r) J_0\left(u\sqrt{m^2+r^2}\right) dr\\
=&2u\int_0^{\frac{u}{2\pi}}\frac{\cos\left(m\sqrt{u^2-4\pi^2y^2}\right)}{\sqrt{u^2-4\pi^2y^2}}h(y)\frac{dy}{y}\\
=& 2u\int_0^u \frac{\cos\left(m\sqrt{u^2-y^2}\right)}{\sqrt{u^2-y^2}} h\left(\frac{y}{2\pi}\right)dy.\qedhere
\end{align*}
\end{proof}

We combine Lemma \ref{lem:fin}, \ref{lem:fin1}, and \ref{lem:fin2} to prove Theorem \ref{app}.
\begin{proof}[Proof of Theorem \ref{app}]
From Lemma \ref{lem:fin2} and \eqref{partial}, we have that
\begin{align*}
\sum_{|n| \gg t_\phi^{1-\delta}}\rho_\phi(n){\rho_\phi(n+m)}G_{t_\phi,h}(n,n+m)
-& \frac{\pi}{t_\phi}\sum_{|n| \gg t_\phi^{1-\delta}}\rho_\phi(n){\rho_\phi(n+m)}k_m\left(\frac{\sqrt{n(n+m)}}{t_\phi}\right)\\
=& O_{A,L}\left( \sum_{|n|\ll Lt_\phi }\frac{|\rho_\phi(n){\rho_\phi(n+m)}|}{\sqrt{n(n+m)}}\|h\|_{W^{A+1,\infty}}\left(t_\phi^{-2+3\delta+3\kappa}+t_\phi^{-\kappa A}\right)\right)\\
=&O_{\epsilon,A,L}\left( \|h\|_{W^{A+1,\infty}}\left(t_\phi^{-2+3\delta+3\kappa+\epsilon}+t_\phi^{-\kappa A+\epsilon}\right)\right).
\end{align*}
Observe that
\begin{align*}
k_{m,h}(u) &= \frac{1}{u}\int_0^{1/u} h\left(\frac{y}{2\pi}\right) \frac{\cos \left(m\sqrt{u^{-2}-y^2}\right)}{\sqrt{u^{-2}-y^2}}\frac{dy}{y}\\
&=\int_0^{1} h\left(\frac{y}{2\pi u}\right) \frac{\cos \left(\frac{m}{u}\sqrt{1-y^2}\right)}{\sqrt{1-y^2}}\frac{dy}{y},
\end{align*}
from which we infer that
\begin{equation}\label{fin:partial}
\frac{\partial^N}{\partial u^N}k_{m,h}(u) \ll_N (m+1)^N u^{-2N} \|h\|_{W^{N,\infty}}.
\end{equation}
When $u<1/(4\pi L)$, we reexpress the integral as follows,
\begin{align*}
k_{m,h}(u) &= \frac{1}{u}\int_0^{1/u} h\left(\frac{y}{2\pi}\right) \frac{\cos \left(m\sqrt{u^{-2}-y^2}\right)}{\sqrt{u^{-2}-y^2}}\frac{dy}{y}\\
&= \frac{1}{u}\int_0^{1/u} h\left(\frac{\sqrt{u^{-2}-x^2}}{2\pi}\right) \frac{\cos \left(m x\right)}{u^{-2}-x^2}dx\\
&= \frac{1}{u}\int_0^{3\pi L} h\left(\frac{\sqrt{u^{-2}-x^2}}{2\pi}\right) \frac{\cos \left(m x\right)}{u^{-2}-x^2}dx\\
&= \frac{1}{u}\int_0^{3\pi L} g\left(\frac{2\pi}{\sqrt{u^{-2}-x^2}}\right) \frac{\cos \left(m x\right)}{u^{-2}-x^2}dx,
\end{align*}
where $g \in C_0^\infty [1/L,L]$ is defined by $g(x)=h(1/x)$. Note that $\|g\|_{W^{N,\infty}} \ll_L \|h\|_{W^{N,\infty}}$. From this we deduce that
\begin{equation}\label{fin:partial2}
\frac{\partial^N}{\partial u^N}k_{m,h}(u) \ll_{N,L}  u^{-N-1} \|g\|_{W^{N,\infty}} \ll_L u^{-N-1} \|h\|_{W^{N,\infty}}
\end{equation}
for $u<1/(4\pi L)$. Combining \eqref{fin:partial} and \eqref{fin:partial2}, we have
\begin{equation}\label{fin:partial1}
\frac{\partial^N}{\partial u^N}k_{m,h}(u) \ll_{N,L} \left\{ \begin{array}{cl}\|h\|_{L^{\infty}}&N=0\\ (m+1)^N u^{-N-1} \|h\|_{W^{N,\infty}} & N \geq 1 \end{array}\right.
\end{equation}
By mean value theorem and \ref{fin:partial1},
\begin{multline*}
k_m\left(\frac{\sqrt{n(n+m)}}{t_\phi}\right) - k_m\left(\frac{|n|}{t_\phi}\right)=k_m'\left(\frac{|n|+\alpha}{t_\phi}\right)\frac{\sqrt{n(n+m)}-|n|}{t_\phi}\\
 = O\left(\frac{m^2 t_\phi\|h\|_{W^{1,\infty}}}{n^2}\right)=O\left(\frac{ t_\phi^{\delta+2\kappa}\|h\|_{W^{1,\infty}}}{n}\right),
\end{multline*}
and
\begin{multline*}
\frac{\pi}{t_\phi}\sum_{|n| \gg t_\phi^{1-\delta}}\rho_\phi(n){\rho_\phi(n+m)}k_m\left(\frac{\sqrt{n(n+m)}}{t_\phi}\right) -\frac{\pi}{t_\phi}\sum_{|n| \gg t_\phi^{1-\delta}}\rho_\phi(n){\rho_\phi(n+m)}k_m\left(\frac{|n|}{t_\phi}\right)\\
=O_{\epsilon, L}(t_\phi^{\delta+2\kappa-1+\epsilon}\|h\|_{W^{1,\infty}}).
\end{multline*}
Theorem \ref{app} follows by combining this estimate with Lemma \ref{lem:fin1}, and then by taking $\delta=1/2$.
\end{proof}

We are ready to prove Corollary \ref{cor122}.
\begin{proof}[Proof of Corollary \ref{cor122}]
Assume that $|t_\phi-T| <T^\theta$ with some $1/3<\theta<1/2-\kappa$. Then by mean value theorem and \eqref{fin:partial},
\[
k_{m,h}\left(\frac{|n|}{t_\phi}\right)- k_{m,h}\left(\frac{|n|}{T}\right)=\left(\frac{|n|}{t_\phi}- \frac{|n|}{T}\right) k_{m,h}\left(\frac{|n|}{T+\alpha}\right) \ll \frac{|n|}{T^{2-\theta}} (m+1) \frac{T^2}{n^2}\|h\|_{W^{1,\infty}} = \frac{ (m+1)T^\theta\|h\|_{W^{1,\infty}}}{|n|},
\]
and so
\begin{align*}
&\frac{\pi}{ t_\phi}\sum_{n \neq 0,-h}\rho_\phi(n){\rho_\phi(n+m)}k_{m,h}\left(\frac{|n|}{t_\phi}\right)\\
=&\frac{\pi}{ T}\sum_{n \neq 0,-h}\rho_\phi(n){\rho_\phi(n+m)}k_{m,h}\left(\frac{|n|}{T}\right)\\
&+O\left(T^{-2+\theta}\|h\|_{L^{\infty}}\sum_{0\leq n \ll LT}{|\rho_\phi(n)|^2}\right)+O\left(\frac{ (m+1)T^\theta\|h\|_{W^{1,\infty}}}{t_\phi}\sum_{0\leq n \ll LT}\frac{|\rho_\phi(n)|^2}{|n|}\right)\\
=&\frac{\pi}{ T}\sum_{n \neq 0,-h}\rho_\phi(n){\rho_\phi(n+m)}k_{m,h}\left(\frac{|n|}{T}\right)+O_{L,\epsilon}\left(T^{-1+\theta+\kappa+\epsilon} \|h\|_{W^{1,\infty}}\right)\\
=&\frac{\pi}{ T}\sum_{n \neq 0,-h}\rho_\phi(n){\rho_\phi(n+m)}k_{m,h}\left(\frac{|n|}{T}\right)+O_{L,\epsilon}\left(T^{-1/2+\epsilon} \|h\|_{W^{1,\infty}}\right),
\end{align*}
where we used \eqref{partial} and summation by parts.

It is immediate from the definition that $k_{m,h}$ is supported in $[0,L]$, but it may not be supported away from $0$. Hence we cannot apply Theorem \ref{thm} directly to the sum
\[
\frac{\pi}{ T}\sum_{n \neq 0,-h}\rho_\phi(n){\rho_\phi(n+m)}k_{m,h}\left(\frac{|n|}{T}\right).
\]
To handle this difficulty, fix a non-negative function $\eta \in C_0^\infty [1/2,2]$ such that $\sum_{j\in \mathbb{Z}} \eta(2^j x) =1$ for all $x>0$. Consider the following sum
\begin{align*}
&\frac{\pi}{T} \sum_{j\in\mathbb{Z}}\sum_{n \neq 0,-h}\rho_\phi(n){\rho_\phi(n+m)}\eta\left(\frac{|n|}{2^j}\right)k_{m,h}\left(\frac{|n|}{T}\right)\\
=&\frac{\pi}{T} \sum_{\substack{j\in\mathbb{Z}\\2^j<T^{1/2}}}\sum_{n \neq 0,-h}\rho_\phi(n){\rho_\phi(n+m)}\eta\left(\frac{|n|}{2^j}\right)k_{m,h}\left(\frac{|n|}{T}\right)\\
&+\frac{\pi}{T} \sum_{\substack{j\in\mathbb{Z}\\T^{1/2}\leq 2^j<2LT}}\sum_{n \neq 0,-h}\rho_\phi(n){\rho_\phi(n+m)}\eta\left(\frac{|n|}{2^j}\right)k_{m,h}\left(\frac{|n|}{T}\right).
\end{align*}
We bound the first sum using \eqref{partial}:
\begin{align*}
&\left|\frac{1}{T} \sum_{\substack{j\in\mathbb{Z}\\2^j<T^{1/2}}}\sum_{n \neq 0,-h}\rho_\phi(n){\rho_\phi(n+m)}\eta\left(\frac{|n|}{2^j}\right)k_{m,h}\left(\frac{|n|}{T}\right) \right|\\
\leq &\frac{1}{T} \sum_{\substack{j\in\mathbb{Z}\\2^j}}\sum_{n \neq 0,-h, |n| <2T^{1/2}}\left|\rho_\phi(n)\rho_\phi(n+m)k_{m,h}\left(\frac{|n|}{T}\right)\right| \\
\ll_{\epsilon,L} &T^{-1/2+\epsilon}\|h\|_{L^\infty}.
\end{align*}
For the second sum, note from \eqref{fin:partial2} that
\[
\frac{\partial^N}{\partial x^N} \left(\eta(x) k_{m,h}\left(\frac{2^j}{T}x\right)\right) \ll_{N,L} 2^{-j}T(1+m)^N \|h\|_{W^{N,\infty}}.
\]
Therefore by Theorem \ref{thm}, we have that
\begin{multline*}
\sum_{|t_\phi-T| <T^{\theta}}\left|\frac{1}{T} \sum_{n \neq 0,-h}\rho_\phi(n){\rho_\phi(n+m)}\eta\left(\frac{|n|}{2^j}\right)k_{m,h}\left(\frac{|n|}{T}\right)-\delta_{0,m}\frac{12}{\pi^2}\frac{2^{j}}{T}\int\eta(x) k_{m,h}\left(\frac{2^j}{T}x\right)dx \right|^2\\
\ll_{\theta, \epsilon} (m^{3/2}+1) 2^j T^{-1+\theta+\epsilon} 2^{-j}T (1+m)^B\|h\|_{W^{B,\infty}} \ll T^{\theta+\kappa (B+3/2)+\epsilon}\|h\|_{W^{B,\infty}},
\end{multline*}
for some $B>0$ depending only on $\theta, \epsilon$.

Combining all these estimates and Theorem \ref{app}, we conclude that
\begin{multline*}
\sum_{|t_\phi-T| <T^\theta}\left|\int_\mathbb{X} P_{m,h}\left(z\right) |\phi\left(z\right)|^2 dV- \delta_{0,m}\sum_{T^{1/2}\leq 2^j<2LT}\frac{12}{\pi}\int\eta\left(\frac{T}{2^j}x\right) k_{m,h}\left(x\right)dx\right|^2\\
\ll_{\theta, \epsilon,L} \left(\sum_{T^{1/2}\leq 2^j<2LT} 1^2\right) T^{\theta+\kappa (B+3/2)+\epsilon}\|h\|_{W^{B,\infty}} + T^{-2\kappa A+2\epsilon+1+\theta }\|h\|_{W^{A,\infty}}\\
\ll_{\epsilon} T^{\theta+\kappa (B+3/2)+2\epsilon}\|h\|_{W^{B,\infty}} + T^{-2\kappa A+2 }\|h\|_{W^{A,\infty}}.
\end{multline*}

To complete the proof of Corollary \ref{cor122} for $1/3<\theta<1/2-\kappa$, we use \eqref{fin:partial2} so that
\[
\sum_{T^{1/2}\leq 2^j<2LT}\frac{12}{\pi}\int\eta\left(\frac{T}{2^j}x\right) k_{0,h}\left(x\right)dx = \frac{12}{\pi}\int_0^\infty k_{0,h}(x)dx + O(T^{-1/2}\|h\|_{L^\infty}).
\]
Now
\begin{align*}
\frac{12}{\pi}\int k_{0,h}(x)dx &= \frac{12}{\pi}\int_0^\infty \int_0^1 h(\frac{y}{2\pi x}) \frac{dydx}{y\sqrt{1-y^2}} \\
&= \frac{12}{\pi}\int_0^\infty \int_0^1 h(u) \frac{dydu}{2\pi u^2\sqrt{1-y^2} }\\
&=\frac{3}{\pi}\int_0^\infty  h(u) \frac{du}{u^2 }
\end{align*}
and
\[
\frac{3}{\pi}\int_{\mathbb{X}} P_{m,h}(z) dV = \delta_{m,0}\frac{3}{\pi}\int_0^\infty h(y) \frac{dy}{y^2}.
\]
Now we have
\begin{multline*}
\sum_{|t_\phi-T| <T^\theta}\left|\int_\mathbb{X} P_{m,h}\left(z\right) |\phi\left(z\right)|^2 dV- \frac{3}{\pi}\int_{\mathbb{X}} P_{m,h}(z) dV\right|^2\\
\ll_{\epsilon} T^{\theta+\kappa (B+3/2)+2\epsilon}\|h\|_{W^{B,\infty}} + T^{-2\kappa A+2 }\|h\|_{W^{A,\infty}},
\end{multline*}
and for given $\epsilon_0>0$, we take $\theta = 1/3+\epsilon_0/100$, $\kappa = \epsilon_0/(200 B+300)$, $\epsilon=\epsilon_0/100$ and $A = (2000 B+3000)/\epsilon_0$ to conclude
\begin{multline*}
\sum_{|t_\phi-T| <T^{1/3}}\left|\int_\mathbb{X} P_{m,h}\left(z\right) |\phi\left(z\right)|^2 dV- \int_{\mathbb{X}} P_{m,h}(z) dV\right|^2 \\\leq \sum_{|t_\phi-T| <T^{\theta}}\left|\int_\mathbb{X} P_{m,h}\left(z\right) |\phi\left(z\right)|^2 dV- \int_{\mathbb{X}} P_{m,h}(z) dV\right|^2
\ll_{\epsilon_0,L} T^{1/3+\epsilon_0}\|h\|_{W^{A,\infty}}.
\end{multline*}
\end{proof}
\subsection{\texorpdfstring{$L^p$}{Lp} restrictions}
In \cite{gzs}, lower bound of the form $\gg 1$ is obtained for
\[
I=\int_0^\infty k(y) |\phi(iy)|^2\frac{dy}{y},
\]
using arithmetic Quantum Unique Ergodicity theorem, under certain assumptions on $k(y)$. We present how one can modify the proof in \cite{gzs} to remove assumptions on $k(y)$ for almost all $\phi$ using Theorem \ref{thm}.
\begin{proof}[Proof of Corollary \ref{cor112}]
Fix a compact geodesic segment $\beta =  \{iy~:~ a< y<b \}$ and let $k(y)$ be a nonnegative compactly supported smooth function supported in $[a,b]$. Assume that $0<k(y) \leq 1$ for $y \in (a,b)$, and that $|T-t_\phi|<T^\theta$.

Fix a function $\alpha \in C_0^\infty \mathbb{R}^+$ such that $\alpha(x) = 0$ for $x<1$, $\alpha(x) = 1$ for $x>2$, and $0 \leq \alpha(x) \leq 1$ for $1 \leq x \leq 2$. As in (59) \cite{gzs}, for $M = t_\phi / \eta$ with the constant $\eta$ to be chosen later, we write
\[
I_1 = 4\sum_{m\geq 1} \sum_{n\geq 1} \rho_\phi(m) \rho_\phi(n) \alpha(m/M) G(m,n)
\]
and
\[
I_2 = 4\sum_{m\geq 1} \sum_{n\geq 1} \rho_\phi(m) \rho_\phi(n) \alpha(m/M)\alpha(n/M) G(m,n)
\]
where
\[
G(m,n) =\cosh \left(\pi t_\phi\right) \int_0^\infty K_{it_\phi }(2\pi m y) K_{it_\phi}(2\pi n y)k(y) dy.
\]
Then we have $I \geq 2I_1 - I_2$. We decompose the sum into three parts; $I_j^D$ is the sum with $m=n$, $I_j^S$ is the sum with $1\leq |m-n|<R$, and $I_j^L$ is the sum with $|m-n|\geq R$.

Following the proof of Proposition 6.4 of \cite{gzs}, we find that
\begin{multline*}
|I_j^L| \ll_{l,a,b} R^{-l+\frac{3}{2}}\sqrt{\eta} \left(\frac{1}{t_\phi} \sum_{n \ll_{a,b} t_\phi} |\rho_\phi (n)|^2\right) + R^{-\frac{1}{2}} \left(t_\phi^{-\frac{3}{2}}\sum_{n \ll_{a,b} t_\phi} |\rho_\phi(n)|^2 \sqrt{n}\right) \\
+R^{-\frac{1}{2}}  \left(t_\phi^{\frac{1}{2}}\sum_{n \geq M} \frac{|\rho_\phi(m)|^2}{m^{3/2}}\right)+o(1),
\end{multline*}
and as in (79) \cite{gzs} this is bounded from above by
\begin{equation}\label{finfin:1}
\ll_{a,b,\eta} R^{-\frac{1}{2}} +o(1)
\end{equation}
uniformly in $\eta$ and $R$.

Now note from Lemma \ref{lem:fin} that
\[
G(m,n) = G_{t_\phi , yk(y)} (m,n)
\]
and hence following the proof of Theorem \ref{app}, we may write
\[
I_1^S = \frac{2\pi}{T}\sum_{1\leq m \leq R} \sum_{n\geq 1} \rho_\phi(n) \rho_\phi(n\pm m) \alpha\left(\frac{n}{M}\right) k_{m,yk(y)} \left(\frac{n}{T}\right) + O_{A,a,b,\epsilon}(\|k\|_{W^{A+1,\infty}}(T^{-1/2+\epsilon}R^4+R^{1-A}T^\epsilon)).
\]
and
\begin{multline*}
I_2^S = \frac{2\pi}{T}\sum_{1\leq m \leq R} \sum_{n\geq 1} \rho_\phi(n) \rho_\phi(n\pm m) \alpha\left(\frac{n}{M}\right)\alpha\left(\frac{n\pm m}{M}\right) k_{m,yk(y)} \left(\frac{n}{t_\phi}\right) \\
+ O_{A,a,b,\epsilon}(\|k\|_{W^{A+1,\infty}}(T^{-1/2+\epsilon}R^3+R^{-A}T^\epsilon)).
\end{multline*}
Since $\alpha((n\pm m)/M) - \alpha (n/M)\ll m/M < \eta R T^{-1}$, we have
\begin{multline*}
I_2^S = \frac{2\pi}{T}\sum_{1\leq m \leq R} \sum_{n\geq 1} \rho_\phi(n) \rho_\phi(n\pm m) \alpha\left(\frac{n}{M}\right)^2 k_{m,yk(y)} \left(\frac{n}{T}\right) \\
+ O_{A,a,b,\epsilon}(\|k\|_{W^{A+1,\infty}}(T^{-1/2+\epsilon}R^3+R^{-A}T^\epsilon+\eta RT^{-1})).
\end{multline*}

For $j=1,2$, for any given $\epsilon>0$ and $1>\theta>1/3$, we know from Theorem \ref{thm} and \eqref{fin:partial1} that there exists sufficiently large $B>0$ depending only on $\epsilon$ and $\theta$ such that
\begin{align*}
&\sum_{|t_\phi -T|<T^\theta}\left|\sum_{1\leq m \leq R} \sum_{n\geq 1} \rho_\phi(n) \rho_\phi(n\pm m) \alpha\left(\frac{n}{M}\right)^j k_{m,yk(y)} \left(\frac{n}{T}\right)\right|^2 \\
\leq &R\sum_{1\leq m \leq R} \sum_{|t_\phi -T|<T^\theta}\left|\sum_{n\geq 1} \rho_\phi(n) \rho_\phi(n\pm m) \alpha\left(\frac{n}{M}\right)^j k_{m,yk(y)} \left(\frac{n}{T}\right)\right|^2\\
=& O_{a,b,\epsilon,\theta,\eta} \left(T^{2+\theta+\epsilon}  R^{B+7/2} \|k\|_{W^{B,\infty}}^2\right).
\end{align*}
We therefore conclude that
\[
\sum_{|t_\phi -T|<T^\theta} |I_j^S(\phi)|^2 \ll_{a,b,\epsilon,\theta,\eta,A}T^{\theta+\epsilon}  R^{B+7/2} \|k\|_{W^{B,\infty}} +\|k\|_{W^{A+1,\infty}}^2(T^{\theta+\epsilon}R^8+R^{-2A+2}T^{1+\theta+\epsilon}).
\]
For given $\epsilon_0>0$, we choose $\theta=1/3+\epsilon_0/100$ $R = T^{\epsilon_0/(100B+400)}$, $\epsilon=\epsilon_0/100$, $A=(1000B+4000)/\epsilon_0$, so that
\[
\sum_{|t_\phi -T|<T^{1/3}} |I_j^S(\phi)|^2 \ll_{k,\epsilon_0,\eta} T^{1/3+\epsilon_0/2},
\]
so by Chebyshev's inequality, all but $O_{k,\epsilon_0,\eta}(T^{1/3+\epsilon_0})$ $\phi$ with $|t_\phi-T|<T^{1/3}$ satisfies
\[
|I_j^S(\phi)| < T^{-\epsilon_0/4}.
\]
Also for such $R$, we infer from \eqref{finfin:1} that $I_j^L = o_{k,\epsilon_0,\theta,\eta}(1)$.

For the diagonal term, (78) \cite{gzs} implies that
\[
I_1^D=I_2^D\sim \int_0^\infty k(y)\frac{dy}{y} +O_{a,b}(\eta^{-1}),
\]
and so we complete the proof of Corollary \ref{cor112} by fixing a sufficiently large $\eta$ depending only on $a,b$.

\end{proof}

\begin{proof}[Proof of Corollary \ref{cor2}]
Fix $1/100>\epsilon_0>0$ and let $\theta =1/3+\epsilon_0/4$ and $\epsilon=\epsilon_0/4$. Let  $\psi \in C_0^\infty (1/2,1)$ be a test function such that $0 \leq \psi(y) \leq 1$ and
\[
\alpha = \int_0^\infty \psi(y)dy \geq 1/4.
\]
Then Theorem \ref{thm} with $m=0$ and $X= \pi T^{1/4}$ implies that there exists a sufficiently large constant $A>0$ depending only on $\theta$ and $\epsilon$ that
\[
\sum_{|t_\phi - T|<T^\theta} \left|\sum_{n} |\rho_\phi (n)|^2 \psi\left(\frac{n}{X}\right)- \frac{12}{\pi^2}\alpha T^{1/4}\right|^2 \ll_{\epsilon,\theta} T^{1+1/4+\theta+\epsilon} \|\psi\|_{W^{A,\infty}}^2 \ll_{\epsilon, \theta}T^{1+1/4+\theta+\epsilon},
\]
hence
\begin{align*}
\sum_{|t_\phi - T|<T^{1/3}} \left|\sum_{n} |\rho_\phi (n)|^2 \psi\left(\frac{n}{X}\right)- \frac{12}{\pi^2}\alpha T^{1/4}\right|^2&\leq
\sum_{|t_\phi - T|<T^{1/3+\epsilon_0/4}} \left|\sum_{n} |\rho_\phi (n)|^2 \psi\left(\frac{n}{X}\right)- 4\alpha T^{1/4}\right|^2 \\
&\ll_{\epsilon_0} T^{1+1/4+1/3+\epsilon_0/2}.
\end{align*}
We infer from Chebyshev's inequality that
\begin{equation}\label{lower}
\left|\sum_{n} |\rho_\phi (n)|^2 \psi\left(\frac{n}{X}\right)- \frac{12}{\pi^2}\alpha T^{1/4}\right|^2<T^{1/2-\epsilon_0/2}
\end{equation}
for all but $\ll_{\epsilon_0}T^{1+1/4+1/3+\epsilon_0/2-(1/2-\epsilon_0/2)}=T^{13/12+\epsilon_0}$ forms. 

For $\phi$ satisfying \eqref{lower}, we have
\[
\sum_{n\leq T^{1/4}} |\rho_\phi (n)|^2 \geq \sum_{n} |\rho_\phi (n)|^2 \psi\left(\frac{n}{X}\right) \geq \frac{3}{\pi^2} T^{1/4}- T^{1/4-\epsilon_0/4} \gg_{\epsilon_0} T^{1/4}.
\]

To prove Corollary \ref{cor2}, recall Equation A.12 of \cite{iw}:
\[
\sum_{T<t_\phi<T+1}|\phi\left(z\right)|^2\left|\sum_{n\leq N} \alpha_n \rho_\phi\left(n\right) \right|^2
\ll_\epsilon N^\epsilon T^\epsilon\left(T\sum_{n \leq N}|\alpha_n|^2+\left(N+N^{1/2}y\right)T^{1/2}\left(\sum_{n \leq N} |\alpha_n|\right)^2\right).
\]
\begin{remark}
Note that the last term in Equation A.12 in the original manuscript \cite{iw} reads $\left(\sum |\alpha_n|^2\right)^2$, and it is a typo. The correct term is $\left(\sum |\alpha_n|\right)^2$, as  above.
\end{remark}
Assume that $z$ is in a fixed compact set $C\subset \mathbb{X}$. Choosing $\alpha_n=\rho_\phi\left(n\right)$ and $N=T^{1/4}$, we have
\begin{align*}
|\phi\left(z\right)|^2 \left(\sum_{n\leq T^{1/4}} |\rho_\phi\left(n\right)|^2 \right)^2 &\ll_{C,\epsilon} T^\epsilon \left(T \sum_{n \leq T^{1/4}}|\rho_\phi(n)|^2+T^{3/4}\left(\sum_{n \leq T^{1/4}}|\rho_\phi(n)|\right)^2\right)\\
&\leq T^\epsilon \left(T \sum_{n \leq T^{1/4}}|\rho_\phi(n)|^2+T^{3/4}\left(\sum_{n \leq T^{1/4}}|\rho_\phi(n)|^2\right)\left(\sum_{n \leq T^{1/4}}1^2\right)\right)\\
&\ll T^{1+\epsilon}\sum_{n \leq T^{1/4}}|\rho_\phi(n)|^2
\end{align*}
and so
\[
|\phi(z)| \ll_\epsilon T^{1/2+\epsilon} \left(\sum_{n \leq T^{1/4}}|\rho_\phi(n)|^2\right)^{-1/2}
\]
Therefore for $\phi$ satisfying
\begin{equation}\label{linfty}
\sum_{n\leq T^{1/4}} |\rho_\phi\left(n\right)|^2 \gg_{\epsilon_0} T^{1/4},
\end{equation}
we have
\[
\sup_{z \in C}|\phi\left(z\right)| \ll_{C,\epsilon,\epsilon_0} t_\phi^{\frac{3}{8}+\epsilon}.
\]
We complete the proof by choosing $\epsilon=\epsilon_0$.
\end{proof}
\subsection{Application to the number of nodal domains}
Finally, we combine Corollary \ref{cor112} and Corollary \ref{cor2} to prove Theorem \ref{thm6}.
\begin{proof}[Proof of Theorem \ref{thm6}]
Fix a geodesic segment $\beta \subset \left\{iy~:~y>0\right\}$, and assume that it is given by $\{iy~:~a<y<b\}$. Let $M_1\left(\phi\right)$ be the maximum of period integrals of $\phi$ taken over segments of $\beta$,
\[
M_1\left(\phi\right)=\sup_{a < \alpha_1 <\alpha_2 <b}\left|\int_{\alpha_1}^{\alpha_2} \phi\left(iy\right) \frac{dy}{y}\right|.
\]
Let $S_\beta(\phi)$ be the number of sign changes of $\phi$ along $\beta$. Denote by $a<\xi_1(\phi) < \xi_2(\phi) < \ldots < \xi_{S_\beta(\phi)}(\phi)<b$ the zeros of $\phi(iy)$ on the interval $(a,b)$ where $\phi(iy)$ changes sign. Put $\xi_0(\phi) = a$ and $\xi_{S_\beta(\phi)+1}=b$. Then we have
\begin{align*}
\int_a^b |\phi(iy)| dy &= \sum_{j=1}^{S_\beta(\phi)+1} \left|\int_{\xi_{j-1}(\phi)}^{\xi_{j}(\phi)}\phi(iy) dy\right|\\
&\leq \sum_{j=1}^{S_\beta(\phi)+1} M_1(\phi),
\end{align*}
hence
\[
\|\phi\|_{L^1\left(\beta\right)} \leq M_1\left(\phi\right) \left(S_\beta\left(\phi\right)+1\right),
\]
From Equation (6) in \cite{gzs}, we have
\[
N^\beta \left(\phi\right) \geq \frac{1}{2}S_\beta\left(\phi\right) +1,
\]
and so it is sufficient to consider $S_\beta\left(\phi\right)$.

Firstly, recall from Section 6.3 of \cite{gzs} that one may bound $M_1\left(\phi\right)$ in terms of integral of Maass $L$-function as follows:
\begin{equation}\label{final:1}
M_1\left(\phi\right) \ll_\epsilon t_\phi^{-1/4+\epsilon}\int_0^{2t_\phi}\left|L\left(\frac{1}{2}+it,\phi\right)\right|\left(1+|t-t_\phi|\right)^{-1/4}\min\left\{1,\frac{1}{t}\right\}dt+e^{-ct_\phi}
\end{equation}
for some constant $c>0$. Assuming that $T<t_\phi<T+1$, we may rewrite \eqref{final:1} as
\begin{equation}\label{final:2}
M_1\left(\phi\right) \ll_\epsilon T^{-1/4+\epsilon}\int_0^{2(T+1)}\left|L\left(\frac{1}{2}+it,\phi\right)\right|\left(1+|t-T|\right)^{-1/4}\min\left\{1,\frac{1}{t}\right\}dt+e^{-cT}.
\end{equation}
Now note that
\begin{align*}
M_1(\phi)^2 \ll_\epsilon & T^{-1/2+2\epsilon}\left(\int_0^{2(T+1)}\left|L\left(\frac{1}{2}+it,\phi\right)\right|\left(1+|t-T|\right)^{-1/4}\min\left\{1,\frac{1}{t}\right\}dt\right)^2+e^{-2cT}\\
\leq& T^{-1/2+2\epsilon}\int_0^{2(T+1)}\left(1+|t-T|\right)^{-1/4}\min\left\{1,\frac{1}{t}\right\}dt\\
&\times \int_0^{2(T+1)}\left|L\left(\frac{1}{2}+it,\phi\right)\right|^2\left(1+|t-T|\right)^{-1/4}\min\left\{1,\frac{1}{t}\right\}dt+e^{-2cT}.
\end{align*}
by Cauchy--Schwartz inequality. Recall an average version of Lindel{\"o}f Hypothesis for Hecke--Maass $L$-functions \cite{jut}
\[
\sum_{T<t_\phi<T+1}|\rho_\phi(1)|^2\left|L\left(\frac{1}{2}+it,\phi\right)\right|^2 \ll_\epsilon \left(T+t^{2/3}\right)^{1+\epsilon}.
\]
Applying \eqref{triv}, we find that for $0<t <2(T+1)$,
\[
\sum_{T<t_\phi<T+1}\left|L\left(\frac{1}{2}+it,\phi\right)\right|^2 \ll_\epsilon T^{1+\epsilon}.
\]
Combining these two estimates, we obtain
\begin{align*}
\sum_{T<t_\phi<T+1}M_1(\phi)^2 \ll_\epsilon & T^{-1/2+2\epsilon}\int_0^{2(T+1)}\left(1+|t-T|\right)^{-1/4}\min\left\{1,\frac{1}{t}\right\}dt\\
&\times \int_0^{2(T+1)}\left(\sum_{T<t_\phi<T+1}\left|L\left(\frac{1}{2}+it,\phi\right)\right|^2\right)\left(1+|t-T|\right)^{-1/4}\min\left\{1,\frac{1}{t}\right\}dt+e^{-cT}\\
\ll_\epsilon &T^{1/2+3\epsilon}\left(\int_0^{2(T+1)}\left(1+|t-T|\right)^{-1/4}\min\left\{1,\frac{1}{t}\right\}dt\right)^2+e^{-cT}.
\end{align*}
Note that
\begin{align*}
&\int_0^{2(T+1)}\left(1+|t-T|\right)^{-1/4}\min\left\{1,\frac{1}{t}\right\}dt \\
=& \int_0^{T/2}\left(1+|t-T|\right)^{-1/4}\min\left\{1,\frac{1}{t}\right\}dt + \int_{T/2}^{2(T+1)}\left(1+|t-T|\right)^{-1/4}\min\left\{1,\frac{1}{t}\right\}dt\\
\ll& T^{-1/4} \int_0^{T/2} \min\left\{1,\frac{1}{t}\right\} dt + T^{-1}\int_{T/2}^{2(T+1)}\left(1+|t-T|\right)^{-1/4}dt\\
\ll& T^{-1/4}\log T,
\end{align*}
from which we conclude that
\[
\sum_{T<t_\phi<T+1}M_1\left(\phi\right)^2 \ll_\epsilon T^\epsilon,
\]
hence
\[
\sum_{|T-t_\phi|<T^{1/3}}M_1\left(\phi\right)^2 \ll_\epsilon T^{1/3+\epsilon}.
\]

Now fix $1/12>\epsilon_0>0$. We infer from Chebyshev's inequality that
\[
M_1(\phi) <t_\phi^{-1/2+\epsilon_0/2}
\]
for all but $O_{\epsilon_0} \left(T^{4/3-\epsilon_0/2}\right)$ forms in $\{\phi~:~|t_\phi-T|<T^{1/3}\}$. Recall from Corollary \ref{cor112} with $\epsilon=2/3$ that all but $O\left(T\right)$ even forms in  $\{\phi~:~|t_\phi-T|<T^{1/3}\}$ satisfy
\[
\|\phi\|_{L^2(\beta)}^2 \gg_{\beta} 1
\]
and from Corollary \ref{cor2} with $\epsilon=\epsilon_0/2$ that all but $O_{\epsilon_0}\left(T^{13/12+\epsilon_0/2}\right)$ forms in $\{\phi~:~|t_\phi-T|<T^{1/3}\}$ satisfy
\[
\|\phi\|_{L^\infty(\beta)} \ll_{\beta,\epsilon_0} t_\phi^{3/8+\epsilon_0/2}
\]
Therefore for all but $O_{\epsilon_0} \left(T^{4/3-\epsilon_0/2}\right)$ even forms in $\{\phi~:~|t_\phi-T|<T^{1/3}\}$ satisfy
\[
S_\beta(\phi)+1 > \frac{\|\phi\|_{L^1(\beta)}}{M_1(\phi)} \geq \frac{\|\phi\|_{L^2(\beta)}^2}{\|\phi\|_{L^\infty(\beta)}M_1(\phi)} \gg_{\beta,\epsilon_0} t_\phi^{1/8-\epsilon_0}.
\]
This completes the proof of Theorem \ref{thm6}.
\end{proof}
\begin{remark}
It is conjectured that $L^\infty$-norm of a Maass form in a fixed compact set is bounded from above by $t_\phi^\epsilon$ for any $\epsilon>0$ \cite{iw}. If this conjecture were true, then it follows automatically that
\begin{equation}\label{lp}
\int_a^b |\phi\left(iy\right)|^p dy \ll_{p,\epsilon} t_\phi^\epsilon.
\end{equation}
Now recall the H\"{o}lder's inequality:
\[
\|f\|_{L^p}^p\|f\|_{L^1}^{p-2}\geq \|f\|_{L^2}^{2\left(p-1\right)}
\]
for $p>2$. In the proof of Theorem \ref{thm6}, we used $p=+\infty$ to obtain a lower bound of the form $\gg_\epsilon t_\phi^{-3/8-\epsilon}$ for the $L^1$-restriction. This can be improved if we know \eqref{lp} for some $p>2$, and would yield a lower bound of the form $\gg_\epsilon t_\phi^{-\epsilon}$, which is essentially optimal.
\end{remark}

\bibliographystyle{alpha}
\bibliography{bibfile}

\end{document}